\documentclass[10pt]{amsart}
\usepackage{amsmath,amsthm,amssymb,amsfonts, eucal, amscd, mathbbol, mathrsfs, mathabx}
\usepackage[shortlabels]{enumitem}
\usepackage{cite}
\usepackage{setspace}
\usepackage[all,cmtip]{xy}{
\usepackage{graphicx}
\usepackage{subfig}
\usepackage{fancyhdr}
\usepackage{latexsym}
\usepackage{fncylab}
\usepackage{epic}
\usepackage{ifthen}
\usepackage[bookmarks,bookmarksnumbered, plainpages=false, pdfpagelabels]{hyperref}
\usepackage{xcolor}
\hypersetup{
 colorlinks   = true, %Colours links instead of ugly boxes
 urlcolor     = green, %Colour for external hyperlinks
 linkcolor    = blue, %Colour of internal links
 citecolor   = red %Colour of citations
}

\usepackage{rotating}
\usepackage{scalefnt}
\usepackage{enumitem}
\setlist{nolistsep}

\parskip = 0.2in
\parindent = 0.0in
\topmargin = 0.0in

\oddsidemargin = 0.0in
\evensidemargin = 0.0in
\textwidth = 6.0in

\newtheorem{thm}{Theorem}[section]
\newtheorem*{thm*}{Main Theorem}
\newtheorem*{thm**}{Theorem}

\newtheorem{cor}[thm]{Corollary}
\newtheorem{lem}[thm]{Lemma}
\newtheorem{prop}[thm]{Proposition}
\theoremstyle{definition}
\newtheorem{defn}[thm]{Definition}
\theoremstyle{definition}

\theoremstyle{definition}
\newtheorem{remark}[thm]{Remark}
\theoremstyle{definition}
\newtheorem{remarks}[thm]{Remarks}
\theoremstyle{definition}
\newtheorem{examples}[thm]{Examples}
\theoremstyle{definition}

\theoremstyle{definition}

\numberwithin{thm}{subsection}

\newcommand{\R}{\ensuremath{\mathbb{R}}}
\newcommand{\N}{\ensuremath{\mathbb{N}}} 

\newcommand{\Z}{{\Bbb Z}} 
\def\p{\partial}
\def\i{\infty}
 
\def\supp{\it{supp}}

\def\diam{\emph{diam}} 
\def\dim{\mathrm{dim}}
\def\fr{\emph{fr}}
 
\def\image{\emph{im}}

\def\cal{\mathcal}
\def\a{\alpha}
\def\b{\beta} 
\def\g{\gamma} 
\def\G{\Gamma}
\def\d{\delta}
\def\D{\Delta} 
\def\e{\epsilon}
\def\k{\kappa} 
\def\l{\lambda}

\def\o{\omega}

\def\H{\mathcal{H}}

\def\F{\mathcal{F}}

\def\cN{\cal{N}}
\def\lip{\mathrm{Lip}}

\def\sp{\mathcal{S}}

\def\Span{\mathcal{S}}
\def\Gr{\mathrm{Gr}}

\def\XXint#1#2#3{{\setbox0=\hbox{$#1{#2#3}{\int}$}
\vcenter{\hbox{$#2#3$}}\kern-.5\wd0}}

\renewcommand{\labelenumi}{(\alph{enumi})}
\renewcommand{\theenumi}{(\alph{enumi})}
\renewcommand{\labelenumi}{\theenumi}
\makeatletter
\renewcommand{\p@enumii}{\theenumi}
\makeatother                                                

\begin{document} 
	\author[J. Harrison \& H. Pugh]{J. Harrison \\Department of Mathematics \\University of California, Berkeley  \\ H. Pugh\\ Mathematics Department \\ Stony Brook University} 
	\title[Cohomological Spanning conditions]{Solutions to the Reifenberg Plateau problem with cohomological spanning conditions}
% \email harrison@math.berkeley.edu
% 	\phone 1-510-326-1644
% 	\fax 1-510-642-8244
% 	\MSC 49 Calculus of variations and optimal control,  28 Measure and Integration and 35 Partial Differential Equations
		
\begin{abstract}
	We prove existence and regularity of minimizers for H\"older densities over general surfaces of arbitrary dimension and codimension in \( \R^n \), satisfying a cohomological boundary condition, providing a natural dual to Reifenberg's Plateau problem. We generalize and extend methods of Reifenberg, Besicovitch, and Adams; in particular we generalize a particular type of minimizing sequence used by Reifenberg (whose limits have nice properties, including lower bounds on lower density and finite Hausdorff measure,) prove such minimizing sequences exist, and develop cohomological spanning conditions. Our cohomology lemmas are dual versions of the homology lemmas in the celebrated appendix by Adams found in Reifenberg's 1960 paper. 
\end{abstract}
			
\maketitle
\section*{Introduction}
	\label{sec:introduction}  	  
	The seminal 1960 work of Reifenberg \cite{reifenberg} was the first to solve a variety of Plateau problems which minimize Hausdorff measure of a collection of surfaces which in some specified sense ``span'' a constrained boundary set \( A \). One can consider \( \frak{m}_A := \inf\{\H^m(X): X \mbox{ spans } A\} \) and hope to find a set \( X_0 \) which spans \( A \) and \( \H^m(X_0) = \frak{m}_A \). One would also like to say something about the structure of the set \( X_0 \), assuming it exists.
	
	Reifenberg used \v{C}ech homology to define spanning sets, as did Almgren \cite{almgrenannals} in 1968. Both solved existence and regularity problems. Although their theories give very satisfying solutions for single component boundaries, homology theory turned out to be inadequate for boundaries consisting of multiple connected components (see Proposition 5.0.1 and Figure 1 in \cite{hpplateau}, as well as \ref{threerings} below.) In \cite{hpplateau}, the authors used linking numbers and differential chains to solve this problem for codimension 2 boundaries in \( \R^n \). De Lellis, Ghiraldin, and Maggi \cite{delellisandmaggi} used our idea of linking numbers to provide a different proof without reference to differential chains. In 2014 \cite{cohomology}, the authors proposed using \v{C}ech cohomology to extend \cite{hpplateau} to arbitrary dimension and codimension and elliptic integrands, although no details were given of the latter. In 2015, De Philippis, De Rosa, and Ghiraldin \cite{ghiraldin} built upon \cite{hpplateau} and \cite{delellisandmaggi} and used a relaxed deformation requirement to solve the problem in arbitrary codimension.  
	
	In this paper we define spanning sets of varying cohomological types and minimize with respect to bounded H\"older\footnote{The reviewer kindly pointed out that our original Lipschitz assumption can be improved to H\"{o}lder continuous.} densities. The variable density problem is a perturbative version of the density one case, and is not the main point of this paper. However, the elliptic problem for anisotropic integrands in \cite{elliptic} builds upon parts of the current work, and so we find it necessary to state our results in generality beyond the density one case.
	 
	We begin with our definition of spanning sets satisfying cohomological conditions: If \( A\subset \R^n \) is compact, \( 2\leq m\leq n \) and \( L \) is a subset of the \( (m-1) \)-st \v{C}ech cohomology of \( A \) with \( \cal{R} \)-module coefficients, we say that a compact set \( X\supset A \) is a \emph{\textbf{surface with coboundary \( \supset L \)}} if the elements of \( L \) do not extend over \( X \). If \( A\subset C \) and \( C \) is convex and compact, let \( \Span(A,C,L,m) \) denote the set of 
surfaces \( X \) with coboundary \( \supset L \) which are contained in \( C \), such that \( \H^m(X\setminus A)<\i \). 
	Let \( \sp^m \) denote \( m \)-dimensional Hausdorff spherical measure and let \( f:C \to [a,b] \) be an \( \a \)-H\"{o}lder continuous function for \( 0<\a\leq 1 \) with \( a > 0 \) and \( b < \i \). Let \( \F^m(E) = \int_E f(x) d\sp^m  \) for all \( \sp^m \)-measurable \( E\subset C \).  We prove:

\begin{thm*}
	If \( \Span(A,C,L,m) \) is non-empty (e.g. if \( \H^{m-1}(A)<\i \) or if \( A \) is bilipschitz homeomorphic to a finite simplicial complex) then the infimum of \( \F^m(X\setminus A) \) among elements \( X \) of \( \Span(A,C,L,m) \) is achieved, and is non-zero. Every such minimizer \( X_0 \) is \( m \)-rectifiable\footnote{In the sense of \cite{mattila}} away from \( A \), has true tangent planes \( \H^m \) a.e. away from \( A \), and contains a surface \( X_0' \) with coboundary \( \supset L \) such that \( X_0' \) contains no proper subset in \( \Span(A,C,L,m) \). Furthermore, there exists an open set \( V \) such that \( \H^m((X_0\cap \mathring{C}) \setminus (A\cup V)) = 0 \) and \( X_0 \cap V \) is a locally H\"older continuously differentiable submanifold of \( \R^n \) with exponent \( \a \).
\end{thm*}

Remarks:

\begin{itemize}
	\item If \( A \) is an orientable manifold and \( L \) consists of the fundamental cocycles of the components of \( A \) (See the definition of \( L^\mathcal{R} \) in \S\ref{section:coboundaries}.\ref{sub:cohomological}) then \( \Span(A,C,L,m) \) subsumes the collections of surfaces \( \mathcal{G} \) and \( \mathcal{G^*} \) of \cite{reifenberg}. In particular, every compact manifold with boundary \( A \) is a surface with coboundary \( \supset L \). If \( A  \) is a topological \( (m-1) \)-sphere, then the surfaces with coboundary \( \supset L \) are precisely the compact sets containing \( A \) which do not retract onto \( A \).
	\item As a corollary, this proves that \( m \)-dimensional compact submanifolds of \( \R^n \) with a non-empty prescribed boundary have a non-zero lower bound on their \( m \)-dimensional Hausdorff measure, depending on the boundary\footnote{The reviewer pointed out that in the smooth case, this corollary can also be obtained by stating the variational problem in the setting of mod 2 currents.}.
	\item The Adams surface \( X \) in Example 8 in the appendix of \cite{reifenberg} admits a retraction onto its frontier \( A \). Lemma \ref{lem:6A} and Lemma 6A of \cite{reifenberg} show that \( X \) does not span \( A \) with respect to any homological or cohomological spanning condition. The linking number test of \cite{hpplateau} also fails: A closed loop is easily drawn that passes once through the middle of the M\"obius strip and the triple M\"obius strip, has linking number one with \( A \), and is disjoint from \( X \).
\end{itemize}

	To prove our main theorem, we modernize and adapt several methods of Reifenberg, Besicovitch and Adams. We use \v{C}ech cohomology instead of homology to define our class of surfaces, and also prove several new results which we use to work with H\"older densities in this paper and elliptic integrands in a sequel \cite{elliptic}. In Lemma \ref{lem:4}, we present a new ``slicing inequality,'' a sharp version of the Eilenberg inequality \cite{eilenberg}.

	Lemma \ref{lem:ratios} is a new ``asymptotic monotonicity'' result which is sufficient to apply Preiss' theorem and deduce our minimizing set is rectifiable. Theorem \ref{thm:haircut} extends a codimension one version of the theorem found in \cite{hpplateau}. We further refine and develop a concept which is used in a deep and fundamental way in \cite{reifenberg}. We originally made use of so-called ``Reifenberg regular sequences'' in \cite{hpplateau} and further extended the concept in \S\eqref{sub:lower_bounds_on_density_ratios}. Reifenberg regular sequences are more general than uniformly quasiminimal sequences\footnote{Since they are Ahlfors regular (see Proposition 4.1 in \cite{davidsemmes})}, which were introduced as \( (\g,\d) \)-restricted sets in \cite{almgren}, and promoted in \cite{davidsemmes1}, \cite{davidAQM} \cite{davidqm}. Limits of Reifenberg regular sequences are not necessarily rectifiable, as are limits of quasiminimal sequences, however limits of Reifenberg regular sequences have non-zero lower bounds on lower density and have finite \( m \)-dimensional Hausdorff measure.

	Our main theorem is proved as follows:
	\renewcommand{\labelenumi}{\arabic{enumi}.}
	\renewcommand{\theenumi}{\arabic{enumi}.}
	\renewcommand{\labelenumi}{\theenumi}
	
	\begin{enumerate}
		\item We apply weak compactness to find a minimizing sequence of measures \( \F^m\lfloor_{X_k\setminus A} \) converging weakly to a measure \( \mu_0 \).
		\item The surfaces \( X_k \) may not converge to \( X_0 \) in the Hausdorff metric, so we provide a haircutting method which modifies these surfaces so that they do converge to \( X_0 \).
		\item We find a Reifenberg regular subsequence as in \cite{hpplateau} to find a non-zero lower bound on lower density of \( X_0 \). 
		\item We analyze measures in slices to prove an asymptotic version of monotonicity. 
		\item From this we deduce the densities exist and are bounded above zero and below infinity. 
		\item Preiss' theorem is used to show that \( X_0\setminus A \) is rectifiable.
		\item It then follows that \( X_0\setminus A \) has true tangent planes almost everywhere, and this leads to a proof of lower semicontinuity of the weighted spherical measure.
	\end{enumerate}

	The paper is organized as follows: Section \ref{section:coboundaries} contains the results necessary to work with surfaces with coboundary. Section \ref{sec:hausdorff_spherical_measure} contains some basic results about Hausdorff spherical measure, including Lemma \ref{lem:4}. In Section \ref{section:isoperimetry}, we generalize some isoperimetry results found in \cite{reifenberg} and use them to prove Theorem \ref{thm:haircut}. The proof of the main theorem starts in Section \ref{sec:minimizing_sequences}, and this section also contains some general results about Reifenberg regular sequences. In Section \ref{sec:monotonicity}, we prove our asymptotic monotonicity result and complete the proof of lower semicontinuity.

\section*{Notation}
	\label{sec:notation}
	If \( X\subset \R^n \),
	\begin{itemize}
		\item \( \fr\,X \) is the frontier of \( X \);
		\item \( \bar{X} \) is the closure of \( X \);
		\item \( \mathring{X} \) is the interior of \( X \);
		\item \( X^c \) is the complement of \( X \);
		\item \( \cN(X,\e) \) is the open epsilon neighborhood of \( X \);
		\item \( B(X,\e) \) is the closed epsilon neighborhood of \( X \);
		\item \( \dim(X) \) is the topological dimension of \( X \);
		\item \( \H^m(X) \) is the \( m \)-dimensional (normalized) Hausdorff measure of \( X \);
		\item \( \sp^m(X) \) is the \( m \)-dimensional (normalized) Hausdorff spherical measure of \( X \);
		\item \( \F^m(X) \) is the integral \( \int_X f\,d\sp^m \), where \( f \) is a given  H\"older function;
		\item \( X^*:= \{p \in X \,|\, \H^m(X \cap B(p,r)) > 0 \text{ for all } r > 0 \} \), where \( m \) is the Hausdorff dimension of \( X \);
		\item \( X(p,r) = X \cap B(p,r) \);
		\item \( x(p,r) = X \cap \fr\,B(p,r) \);
		\item \( \a_m \) is the Lebesgue measure of the unit \( m \)-ball in \( \R^m \);
		\item \( \Gr(m,n) \) is the Grassmannian of un-oriented \( m \)-planes through the origin in \( \R^n \);
	\end{itemize}
	
	\newpage

\section{Coboundaries}\label{section:coboundaries}
\subsection{Cohomological spanning condition}
	\label{sub:cohomological}

	Let \( 1\leq m\leq n \) and \( A\subset \R^n \). If \( \mathcal{R} \) is a commutative ring and \( G \) is a \( \mathcal{R} \)-module, let \( H^{m-1}(A)=H^{m-1}(A;G) \) (resp. \( \tilde{H}^{m-1}(A)=\tilde{H}^{m-1}(A;G) \)) denote the \( (m-1) \)-st (resp. \( (m-1) \)-st reduced\footnote{Let us agree for notational purposes that \( \tilde{H}^0(\emptyset)=0 \) and that the inclusion \( \iota(Y,\emptyset) \) of \( \emptyset \) into any set \( Y \) induces the zero homomorphism \( \iota(Y,\emptyset)^*: \tilde{H}^0(Y)\to \tilde{H}^0(\emptyset) \).}) \v{C}ech cohomology group with coefficients in \( G \). If \( X\supset A \), and \( \iota = \iota(X,A) \) denotes the inclusion mapping of \( A \) into \( X \), let \( K^*(X,A) \) denote the complement in \( \tilde{H}^{m-1}(A) \) of the image of \( \iota^*:\tilde{H}^{m-1}(X)\to \tilde{H}^{m-1}(A) \). Call \( K^*(X,A) \) the \emph{\textbf{(algebraic) coboundary\footnote{In the spirit of Reifenberg and Adams' terminology ``algebraic boundary.''} of \( X \) with respect to \( A \).}}

	Let \( L \subset \tilde{H}^{m-1}(A)\setminus \{0\} \). We say that \( X \) is a \emph{\textbf{surface with coboundary\footnote{After writing this paper, the authors found that this definition was known to Fomenko in \cite{fomenkocohomology}, but that he had not defined the collection \( L^\mathcal{R} \) nor proved most of the results in \S\ref{sub:cohomological}.\ref{sub:reifenberg_appendix}.} \( \supset L \)}} if \( K^*(X,A) \supset L \); in other words, if \( L \) is disjoint from the image of \( \iota^* \).

	For example, if \( L=\emptyset \), then every \( X \supset A \) is a surface with coboundary \( \supset L \). If \( L\neq \emptyset \) and \( X \) is a surface with coboundary \( \supset L \), then \( X \) does not retract onto \( A \). If \( L= \tilde{H}^{m-1}(A)\setminus \{0\} \), then \( X \) is a surface with coboundary \( \supset L \) if and only if \( \iota^* \) is trivial on \( \tilde{H}^{m-1}(X) \).

	If \( A \) is homeomorphic to an \( (m-1) \)-sphere, \( \mathcal{R}=G=\Z \), \( L\simeq \{1,-1\} \) is the set of generators of \( \tilde{H}^{m-1}(A)\simeq \Z \) and \( X\supset A \) is compact with \( \H^m(X)<\i \), then \( X \) is a surface with coboundary \( \supset L \) if and only if \( X \) does not retract onto \( A \). This is due to a theorem of Hopf \cite{hurewicz}.

	More generally, if \( G=\mathcal{R} \) and \( A \) is an \( (m-1) \)-dimensional closed \( \mathcal{R} \)-orientable (topological) manifold, then there is a canonical choice for \( L \), denoted \( L^{\mathcal{R}}=L^\mathcal{R}(A) \): Let \( A_i, i=1,\dots,k \) denote the components of \( A \), and for each \( i \), let \( L_i \) denote the image under the natural linear embedding \( H^{m-1}(A_i)\hookrightarrow H^{m-1}(A)\simeq \oplus_i H^{m-1}(A_i) \) of the \( \mathcal{R} \)-module generators of \( H^{m-1}(A_i)\simeq \mathcal{R} \). If \( m>1 \), let \( L^\mathcal{R}=\cup_i L_i \). If \( m=1 \), define \( L^\mathcal{R} \) to be the projection of \( \cup_i L_i \) onto the reduced cohomology \( \tilde{H}^0(A) \).

	A primary reason for considering the set \( L^\mathcal{R} \) is the following: If \( X \) is a compact \( \mathcal{R} \)-orientable manifold with boundary \( A \), then \( X \) is a surface with coboundary \( \supset L^\mathcal{R} \) (Theorem \ref{thm:manifoldspans}.) If \( \mathcal{R}=\Z \), then \( X \) need not be orientable. In fact, if \( X \) is any compact set which can be written as the union of \( A \) and an increasing union of a sequence of compact manifolds with boundary \( X_i \), such that \( \p X_i\cup A=\p B_i \) for a sequence \( \{B_i\} \) of compact manifolds which tend to \( A \) in Hausdorff metric, then \( X \) is a surface with coboundary \( \supset L^\Z \) (Theorem \ref{thm:flatspans}.) When \( n=3, m=2 \), this is the class of surfaces \( \mathcal{G} \) found in \cite{reifenberg}.

	Another feature of \( L^\mathcal{R} \) is the following gluing property: Suppose \( A=A_1\cup\cdots\cup A_k \), where \( A_1,\dots, A_k \) are \( (m-1) \)-dimensional closed \( \mathcal{R} \)-orientable manifolds, and every non-empty intersection of the \( A_i \)'s is also a \( (m-1) \)-dimensional closed manifold. If for each \( i=1,\dots,k \), \( X_i \) is a surface with coboundary \( \supset L^\mathcal{R}(A_i) \), then \( X=\cup_i X_i \) is a surface with coboundary \( \supset L^\mathcal{R}(A) \) (Proposition \ref{prop:unionspans}.)

	If \( A \) is a \( (n-2) \)-dimensional closed oriented manifold, then by Alexander duality \( X \) is a surface with coboundary \( \supset L^\Z \) if and only if \( X \) intersects every embedding \( \gamma: \amalg_{j=1}^l S^1 \to \R^n\setminus A \), \( l\in \N \), such that the linking number \( L(\gamma, A_i) \) with some component \( A_i \) of \( A \) is \( \pm 1 \), and such that \( L(\gamma,A_j)=0 \) for \( j\neq i \).

	More generally, if \( A \) is any compact subset of \( \R^n \), we can view \( A \) as a compact subset of the \( n \)-sphere. Then by Alexander duality, the choice of \( L \) is equivalent to the choice of a subset \( S \) of \( \tilde{H}_{n-m-1}(S^n\setminus A) \). A compact set \( X \) is a surface with coboundary \( \supset L \) if and only if \( X \), viewed as a subset of \( S^n \), intersects the support of every singular chain representing an element \( S \). 

	It follows that if \( X \) is a surface with coboundary \( \supset L^\mathcal{R} \), then \( A\subset \overline{X\setminus A} \).

	If \( C\supset A\), let \( \Span(A,C,L,m) \) denote the collection of compact surfaces \( X \subset C \) such that \( X \) is a surface with coboundary \( \supset L \) and \( \H^m(X\setminus A) < \i \). 
\renewcommand{\labelenumi}{(\alph{enumi})}
\renewcommand{\theenumi}{(\alph{enumi})}
\renewcommand{\labelenumi}{\theenumi}

	\begin{examples} \mbox{}
		\begin{enumerate}
			\item\label{threerings}
			If \( A\subset \R^3 \) is the union of three stacked circles, explicitly \( A=\{(x,y,z)\in\R^3: x^2+y^2=1, z\in\{-1,0,1\} \} \), then the surfaces \( X_1=\{(x,y,z)\in \R^3 : x^2+y^2=1, -1\leq z \leq 0 \}\cup \{(x,y,z)\in \R^3 : x^2+y^2\leq 1, z=1 \}\), \( X_2=\{(x,y,z)\in \R^3 : x^2+y^2=1, 0\leq z \leq 1 \}\cup \{(x,y,z)\in \R^3 : x^2+y^2\leq 1, z=-1 \}\) and \( X_3=\{(x,y,z)\in \R^3 : x^2+y^2=1, -1\leq z \leq 1 \} \) are all surfaces with coboundary \( \supset L^\Z \). One can replace the cylinders with catenoids, and move the circles of \( A \) up or down, in which case any of \( X_1 \), \( X_2 \) or \( X_3 \) could be an area minimizer in \( \Span(A,\R^3,L^\Z,2) \), depending on the distance between the circles of \( A \).

			\item
			If \( A\subset \R^3 \) is a standard \( 2 \)-torus given parametrically by \( x(\theta,\phi)=(R+r\cos\theta)\cos\phi \), \( y(\theta,\phi)=(R+r\cos\theta) \) and \( z(\theta,\phi)=r\sin\theta \), and \( L\subset H^1(A;\Z) \) consists of a single element, the class of the cocycle dual to a longitudinal circle \( \phi=\mathrm{const.} \), then \( X\supset A \) is a surface with coboundary \( \supset L \) if and only if \( X \) contains a longitudinal disk. If one replaces the minor radius \( r \) with a function \( r(\phi) \), then the set \( A\cup D \), where \( D \) is the longitudinal disk at the narrowest part of \( A \), will be an area minimizer in \( \Span(A,\R^3,L,2). \)
		\end{enumerate}
	\end{examples}

	This definition is the natural dual of the definition of a ``surface with boundary \( \supset L \)'' \cite{reifenberg} (see also \cite{almgrenannals}.) Recall \( X \) is a \emph{\textbf{(Reifenberg) surface with (algebraic) boundary \( \supset L \)}} if \( L \) is a subgroup of the kernel of \( \iota_*:H_{m-1}(A)\to H_{m-1}(X) \), this kernel being the \emph{\textbf{algebraic boundary}} of \( X \). Given a choice of \( G \) and \( L \), we call the collection of surfaces with boundary \( \supset L \), a \emph{\textbf{Reifenberg collection.}} Reversing the variance has a number of advantages:
	\begin{itemize}
		\item The group \( G \) may be any \( \mathcal{R} \)-module, not just a compact abelian group as in \cite{reifenberg} or a finitely generated group as in \cite{almgrenannals}.  
		\item The collection of non-retracting surfaces in Theorem 2 of \cite{reifenberg} is achieved as a single collection, namely \( \Span(A\simeq S^{m-1},\R^n,L^\Z,m) \);
		\item The sets \( X_1, X_2, X_3 \) in Example \ref{threerings} above are all surfaces with coboundary \( \supset L^\Z \), but the only Reifenberg collections containing all three correspond to the trivial subgroup \( L=\{0\} \), in which case every set \( X\supset A \) is a surface with boundary \( \supset L \). (See \cite{hpplateau} Proposition 5.0.1.)  
		\item There is a canonical choice of \( L \) in the case that \( A \) is a compact oriented \( (m-1) \)-dimensional manifold, namely the subset \( L^\Z \), and the collection \( \Span(A,\R^n,L^\Z,m) \) is well-behaved and large as described above. In particular, \( \Span(A,\R^3,L^\Z,2) \) contains Reifenberg's class \( \mathcal{G} \), and when \( A \) is homeomorphic to a \( (m-1) \)-sphere, \( \Span(A,\R^n,L^\Z,m) \) contains Reifenberg's class of non-retracting surfaces, \( \mathcal{G}^* \).
	\end{itemize}
	
	\begin{figure}[h]
	  \centering
	    \includegraphics[width=.5\textwidth]{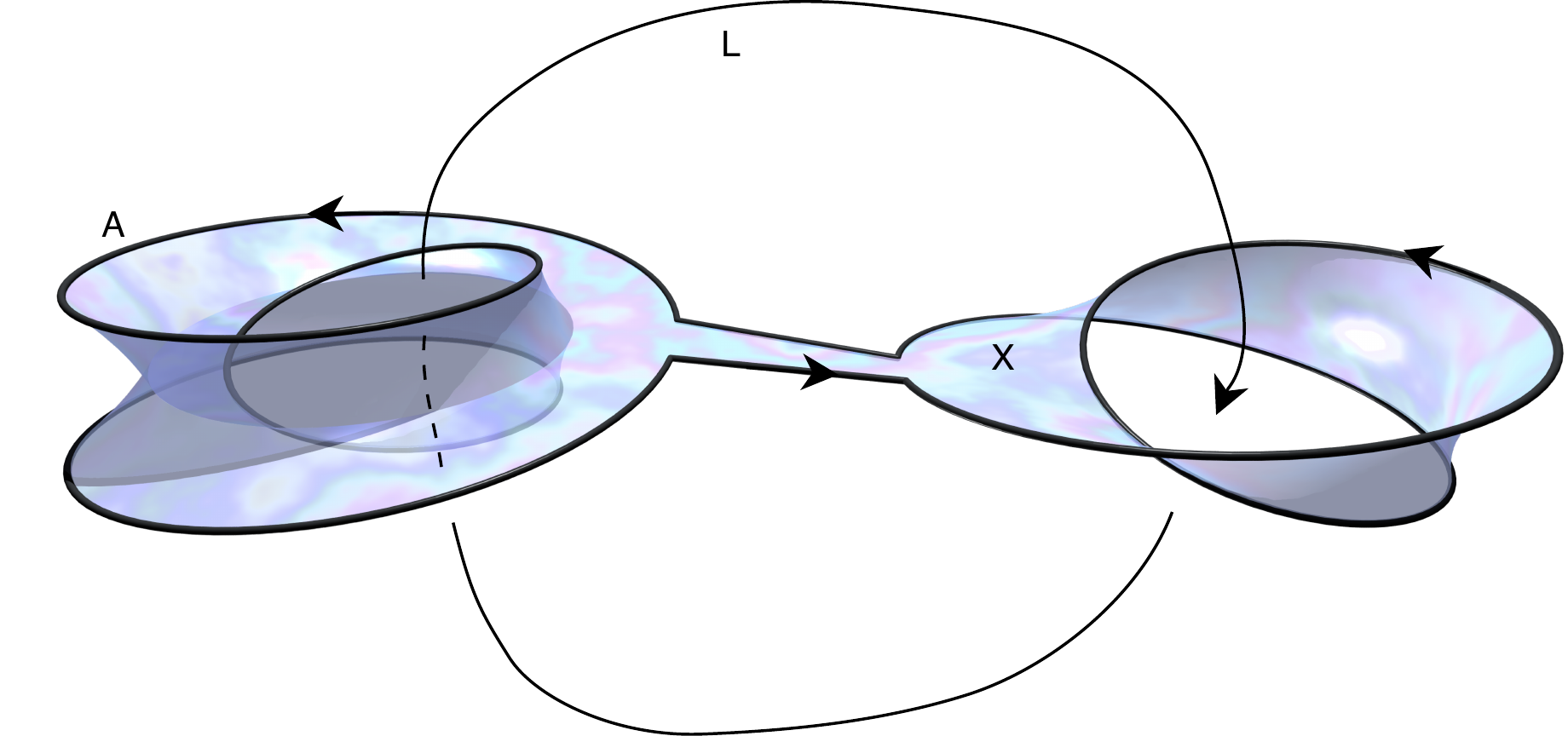}
	  \caption{The Adams surface \( X \) \cite{reifenberg} admits a retraction onto the curve \( A \). It therefore does not span \( A \) with respect to any condition invariant under continuous mappings keeping \( A \) fixed. This includes all homological and cohomological spanning conditions, as well as the linking number tests in \cite{plateau10} \cite{hpplateau}. The figure shows a simple closed curve \( L \) that links the curve \( A \) once, but which is disjoint from \( X \).}
	  \label{fig:AdamsLink}
	\end{figure}

	\subsubsection{Some open questions}
		\label{ssub:some_open_questions}

		\begin{enumerate}
			\item For what choice of \( L\subset H_m(X) \) is it true that if \( X \) is a surface with boundary \( \supset L \), then \( A \subset \overline{X\setminus A} \)? Same question for coboundary.
			\item  If \( m=n-1 \) and \( A \) is a \( (m-1) \)-dimensional orientable compact manifold, the condition that \( X \) is a surface with coboundary \( \supset L^\Z \) is slightly relaxed from the definition of ``span'' using linking numbers in (\cite{plateau10} \cite{hpplateau},) since in the linking number definition, \( L^\Z \) need only be disjoint from the image of those cocycles which are Alexander dual to cycles represented by embedded circles, and not sums of such. In this vein, one can modify the definition of surface with coboundary \( \supset L \) so that \( L \) need only be disjoint from those elements of \( H^k(A) \) which extend over \( X \) as cocycles Alexander dual to cycles representable by manifolds of a given topological type. However, this definition seems difficult to work with. For example, compare Lemma \ref{lem:6A} with Theorem 5.0.6 of \cite{hpplateau}. To what extent is it possible to restrict the topological type of these cycles? See \cite{ghiraldin}.
			\item If \( X \) is a surface with algebraic boundary \( K \), how can one determine the algebraic coboundary \( K^* \) of \( X \)? Same question with \( K \) and \( K^* \) reversed.
			\item If one replaces sets with pairs, the definition can be repeated with relative cohomology: A pair \( (X,Y)\supset (A,B) \) is a surface with coboundary \( \supset L \) if \( L \) is disjoint from the image of \( \iota^*:\tilde{H}^{m-1}(X,Y)\to \tilde{H}^{m-1}(A,B) \). Is this definition useful for working with surfaces which partially span their boundaries?
	 	\end{enumerate}

	\subsection{Cohomological spanning lemmas}
		\label{sub:reifenberg_appendix}

		We now produce a sequence of lemmas, many of whose statements are dual to those found in the appendix of \cite{reifenberg}. We do not assume sets are compact, unless the assumption is made explicit in the lemma. 

		\begin{lem}\label{lem:1A}
			\( K^*(A,A) = \emptyset \).
		\end{lem}

		\begin{proof}
			The identity map on \( \tilde{H}^{m-1}(A) \) is surjective.
		\end{proof}

		\begin{lem}\label{lem:2A}
			If \( X \) is contractible and \( A \subset X \), then \( K^*(X,A) = \tilde{H}^{m-1}(A)\setminus \{0\} \).
		\end{lem}

		\begin{proof}
			By homotopy invariance, \( X \) has the reduced cohomology of a point.
		\end{proof}

		\begin{lem}\label{lem:3A}
			Suppose \( X\supset A \) and \( X = \cup_{i=1}^N X_i \) where the \( X_i \) are disjoint, closed, and contractible. If \( m>1 \), then \( K^*(X,A) = H^{m-1}(A)\setminus \{0\} \). 
		\end{lem}
		
		\begin{proof}
			Let \( A_i=A\cap X_i \). By E-S Ch. I Thm. 13.2c, \( H^{m-1}(X) \simeq \oplus_{i=1}^N H^{m-1}(X_i) \) and \( H^{m-1}(A)\simeq \oplus_{i=1}^N (A_i) \). Moreover, the square
			\begin{align}
				\xymatrix{H^{m-1}(X) \ar@{->}[r]^-{\simeq} \ar@{->}[d]^{\iota(X,A)^*}  & \oplus H^{m-1}(X_i) \ar@{->}[d]^{\oplus \iota(X_i,A_i)^*} \\
				\tilde{H}^{m-1}(A)  \ar@{->}[r]^-{\simeq} &  \oplus H^{m-1}(A_i) }	
			\end{align}
			commutes since it does so for each summand of \( \oplus H^{m-1}(A_i) \). We may then apply Lemma \ref{lem:2A}.
		\end{proof}
		
		\begin{lem}\label{lem:6A}
			Suppose \( g:(X,A) \to (Y,B) \) is continuous. Let \( L_A \subset \tilde{H}^{m-1}(A)\setminus \{0\} \) and \( L_B = (g|_A^*)^{-1}(L_A) \). If \( X \) is a surface with coboundary \( \supset L_A \), then \( Y \) is a surface with coboundary \( \supset L_B \).
		\end{lem}

		\begin{proof}
			The proof is evident from the commutativity of the following square:
			\begin{align}
				\xymatrix{\tilde{H}^{m-1}(X) \ar@{<-}[r]^{g^*} \ar@{->}[d]^{\iota(X,A)^*}  & \tilde{H}^{m-1}(Y) \ar@{->}[d]^{\iota(Y,B)^*} \\
				\tilde{H}^{m-1}(A)  \ar@{<-}[r]^{(g|_A)^*} & \tilde{H}^{m-1}(B). }	
			\end{align}
		\end{proof}

		\begin{lem}\label{lem:7A}
			Suppose \( X \) is a surface with coboundary \( \supset L \). If \( X \subset Y \), then \( Y \) is also a surface with coboundary \( \supset L \).
		\end{lem}

		\begin{proof}
			The inclusion \( A\hookrightarrow Y \) factors through \( X \), so the image of \( \iota(Y,A)^* \) is contained in the image of \( \iota(X,A)^* \).
		\end{proof}

		\begin{lem}\label{lem:8A}
			Let \( m=n \) and suppose \( A \) is the unit sphere in \( \R^n \). If \( X\supset A \) contains the closed unit ball, then \( K^*(X,A) = \tilde{H}^{n-1}(A)\setminus \{0\} \). If \( X\supset A \) does not contain the closed unit ball, then \( K^*(X, A) = \emptyset \).
		\end{lem}

		\begin{proof}
			If \( X \) contains the closed unit ball \( B \), then Lemmas \ref{lem:2A} and \ref{lem:7A} prove the first statement. If \( X \) does not contain \( B \), then \( A \) is a retract of \( X \), and so \( \iota(X,A)^* \) must be surjective. 
		\end{proof}

		\begin{lem}\label{lem:10A} 
			Suppose \( f:I \times Y \to X \) is a continuous map. Set \( A_0 = f(\{0\} \times Y), \quad A_1 = f(\{1\} \times Y), \) and \( A = A_0 \cup A_1 \). Write \( f_0:= f\lfloor_{\{0\} \times Y} \) and \( f_1 :=  f\lfloor_{\{1\} \times Y} \). Suppose \( f_0 \) is a homeomorphism from \( \{0\}\times Y \) to \( A_0 \), and that we are given a subset \( L_0 \subset \tilde{H}^{m-1}(A_0)\setminus \{0\} \). Then there exists a subset \( L_1 \subset \tilde{H}^{m-1}(A_1)\setminus\{0\} \) satisfying two properties: \[ K^*(X,A) \cup (\iota(A,A_0)^*)^{-1}(L_0)  = K^*(X,A) \cup (\iota(A,A_1)^*)^{-1}(L_1) \] and if \( X \) is a surface with coboundary \( \supset L_0 \), then \( X \) is a surface with coboundary \( \supset L_1 \). 
		\end{lem}

		\begin{proof} 
			Define \( g: \{0\} \times Y \to \{1\} \times Y \) where \( g(0,y) = (1,y) \). Let \( L_0' := f_0^*(L_0) \subset \tilde{H}^{m-1}(\{0\} \times Y) \) and \( L_1':= (g^*)^{-1}(L_0') \in \tilde{H}^{m-1}(\{1\} \times Y) \). Finally define \( L_1 = (f_1^*)^{-1}(L_1') \).

			Let \( h \in (\iota(A,A_0)^*)^{-1}(L_0) \) and suppose \( h\notin K^*(X,A) \). That is, suppose \( h=\iota(X,A)^*(x) \) for some \( x\in  \tilde{H}(X) \). We want to show \( \iota(A,A_1)^*(h) \in L_1 \). That is, \( f_1^*\iota(A,A_1)^*(h) \in  L_1' \). In other words, \(g^* f_1^*\iota(A,A_1)^*(h) \in L_0' \), or \[(f_0^*)^{-1}g^* f_1^*\iota(A,A_1)^*\iota(X,A)^*(x) \in  L_0 .\] Since we have assumed \( h=\iota(X,A)^*(x)\in  (\iota(A,A_0)^*)^{-1}(L_0) \), it suffices to show \[ \iota(A,A_0)^*\iota(X,A)^*(x) = (f_0^*)^{-1}g^* f_1^*\iota(A,A_1)^*\iota(X,A)^*(x), \] which is verified by the fact that \( \iota(X,A_0) f_0 \) and \( \iota(X,A_1) f_1 g \) are homotopic. The other containment is proved in a similar manner.

	 		For the last assertion, suppose \( h_1 \in L_1 \) and \( h_1 =  \iota(X, A_1)^*(x) \) for some \( x \in  \tilde{H}^{m-1}(X) \). By definition of \( L_1 \), we know  \( h_0 = (f_0^*)^{-1}g^* f_1^*(h_1) \in L_0 \). The inclusion map \( \iota(X,A_0) \) is homotopic to \( \iota(X,A_1) \circ f_1\circ g \circ f_0^{-1} \) via the homotopy \( \iota(X,A_t) \circ f_t \circ g_t \circ f_0^{-1} \) where \( A_t = f(\{t\} \times Y) \), \( f_t:= f\lfloor_{\{t\} \times Y} \) and \( g_t(0,y) = (t,y) \). Thus \( \iota(X,A_0)^* = (\iota(X,A_1) \circ f_1 \circ g \circ f_0^{-1})^* = (f_0^*)^{-1}g^* f_1^*\iota(X,A_1)^* \). 
	Then \( h_0 = (f_0^*)^{-1}g^* f_1^*(h_1) = (f_0^*)^{-1}g^* f_1^*\iota(X, A_1)^*(x) = \iota(X,A_0)^*(x) \), contradicting our assumption that \( X \) is a surface with coboundary \( \supset L_0 \).
		\end{proof}

		It follows from Lemma \ref{lem:7A} that if \( Z \) is a surface with coboundary \( \supset L_0 \), then \( Z \cup f(I \times Y) \) is a surface with coboundary \( \supset L_1 \). 

		\begin{lem}\label{lem:11A}
			Suppose \( X = \cup_{r=1}^N X_r \), \( A \subset X \), and \( A_r \subset X_r \) for each \( r \). Let \( B = A \cup_r A_r \). For each \( r \), let \( L_r \subset \tilde{H}^{m-1}(A_r)\setminus \{0\} \) and suppose \( X_r \) is a surface with coboundary \( \supset L_r \). Suppose \( L \subset \tilde{H}^{m-1}(A)\setminus \{0\} \) satisfies 
			\begin{equation}\label{eq:E}
		 		(\iota(B,A)^*)^{-1}(L) \subset \cup_r (\iota(B,A_r)^*)^{-1}(L_r).
			\end{equation}
			Then \( X \) is a surface with coboundary \( \supset L \).
		\end{lem}

		\begin{proof}
			Let \( k\in \tilde{H}^{m-1}(X) \). Then \( \iota(X,A)^* k = \iota(B,A)^* i(X,B)^* k \). By assumption, it suffices to show that \( \iota(X,B)^* k \) is not contained in \( (\iota(B, A_r)^*)^{-1} (L_r) \) for each \( r \). In other words, it suffices to show that \( \iota(X,A_r)^* k \) is not contained in \( L_r \), or equivalently, \( \iota(X_r,A_r)^*\iota(X,X_r)^* k \) is not contained in \( L_r \). Indeed, this is true since the image of \( \iota(X_r,A_r)^* \) is disjoint from \( L_r \) by assumption.
		\end{proof}

		\begin{lem}\label{lem:12A}
		 	Under the same assumptions of Lemma \ref{lem:11A}, suppose further that \( X_r \) and \( A_r \) are compact for each \( r \), that \( A \cap X_r \subset A_r \) for each \( r \), and that \( X_r \cap X_s = A_r \cap A_s \) for \( r \ne s \). Then \[ K^*(X,A) = \{x\in \tilde{H}^{m-1}(A): (\iota(B,A)^*)^{-1}(x)\subset \cup_r (\iota(B,A_r)^*)^{-1}(K^*(X_r,A_r))\}. \]
		\end{lem} 

		\begin{proof}
			By Lemma \ref{lem:11A}, \( \{x\in \tilde{H}^{m-1}(A): (\iota(B,A)^*)^{-1}(x)\subset \cup_r (\iota(B,A_r)^*)^{-1}(K^*(X_r,A_r))\}\subset K^*(X,A) \). To show the reverse inclusion, we chase the diagram below. The unlabeled maps are given by inclusions, the rows and column are exact (E-S Ch. I Thm. 8.6c,) and the isomorphism is due to excision: Assuming \( N=2 \), the isomorphism follows from E-S Ch. I Thm. 14.2c and Ch. X Thm. 5.4. The general case follows from induction on \( N \). The triangle commutes by functoriality, the top ``square'' commutes because \( \delta \) is a natural transformation, and the bottom square commutes because it does so on each summand of \( \oplus H^m(X_r,A_r). \) Thus, the diagram commutes.

			Let \( x \in K^*(X,A) \) and suppose \( p \in (\iota(B,A)^*)^{-1}(x) \). Suppose there is no \( r \) such that \( \iota(B,A_r)^*(p) \in K^*(X_r,A_r) \). Then \( y = \oplus (\iota(B,A_r)^*)(p) \in \image \oplus \iota(X_r,A_r)^* \). Since the bottom row of the diagram is exact, \( (\oplus \d)(y) = 0 \), hence \( \d p = 0 \), hence \( \d x=0 \). But the left column is exact, and this gives a contradiction, since by assumption \( x \) is not in the image of \( \iota(X,A)^* \).
			\begin{align}
				\xymatrix{&H^m(X,A) \ar@{<-}[d]^{\d}  \ar@{<-}[rrdd] &\quad & \quad  \\ &\tilde{H}^{m-1}(A) \ar@{<-}[d]  \ar@{<-}[dr]  &\quad & \quad \\
				&\tilde{H}^{m-1}(X)  \ar@{->}[r]   & \tilde{H}^{m-1}(B) \ar@{->}[d] \ar@{->}[r]^{\d} & H^m(X,B) \ar@{->}[d]^{\cong}\\
				&\oplus \tilde{H}^{m-1}(X_r)\ar@{->}[r]  &\oplus \tilde{H}^{m-1}(A_r) \ar@{->}[r]^{\oplus \d} &\oplus H^m(X_r,A_r).}
			\end{align}
		\end{proof}
  
		\begin{lem}\label{lem:13}
			Suppose \( A, X \) and \( C \) are compact, \( X \) is a surface with coboundary \( K^*(X,A) \supset L \) and \( \mathring{C}\cap A=\emptyset \). If \( Y\supset X\cap \fr\, C \) is a surface with coboundary \( \supset K^*(X\cap C, X\cap \fr\, C) \), then \( (X\setminus \mathring{C}) \cup Y \) is a surface with coboundary \( \supset L \).
		\end{lem}

		\begin{proof}
			Let \( X_1 = X\cap C \), \( A_1=X\cap \fr\, C \), \( X_2=X\setminus \mathring{C} \), and \( A_2=A\cup A_1 \). Let \( L_1=K^*(X_1,A_1) \) and \( L_2=K^*(X_2, A_2) \). By Lemma \ref{lem:12A}, \[ K^*(X,A)= \left\{x\in \tilde{H}^{m-1}(A): (\iota(A_2,A)^*)^{-1}(x)\subset \left((\iota(A_2,A_1)^*)^{-1} L_1\right)\cup L_2\right\}. \] Now apply Lemma \ref{lem:11A}, using the set \( Y \) in place of \( X_1 \). The result follows, since \( L\subset K^*(X,A) \).
		\end{proof}

		\begin{lem}\label{lem:13A}
			Suppose \( A = A_1 \cup A_2 \) where \( A_1 \) and \( A_2 \) are compact. Let \( D = A_1 \cap A_2 \) and suppose \( B\supset D \) is compact. Let \( m\geq 2 \). Suppose the homomorphism \( \iota(B,D)^*: \tilde{H}^{m-2}(B)\to \tilde{H}^{m-2}(D) \) is zero. Then \[ (\iota(A\cup B,A)^*)^{-1} (H^{m-1}(A)\setminus \{0\}) \subset \cup_{i=1,2}(\iota(A\cup B, A_i\cup B)^*)^{-1} (H^{m-1} (A_i\cup B )\setminus \{0\}). \]
		 \end{lem}

		\begin{proof} 
			Suppose \( D \) is non-empty. The map \( (A,A_1,A_2) \to (A\cup B, A_1\cup B,A_2\cup B) \) is a map of compact, and hence proper triads (E-S Ch. X Thm 5.4,) and thus carries the reduced Mayer-Vietoris sequence of the second into the first (E-S 15.4c.) Chase the resulting commutative diagram, observing that \( \cup_{i=1,2}(\iota(A\cup B, A_i\cup B)^*)^{-1} (H^{m-1} (A_i\cup B )\setminus \{0\}) = (\iota_1^*,\iota_2^*)^{-1}(H^{m-1}(A_1\cup B)\oplus H^{m-1}(A_2\cup B)\setminus \{0\}) \):
			\begin{align}
				\xymatrixcolsep{3pc}\xymatrix{\tilde{H}^{m-2}(B) \ar@{->}[d]^0 \ar@{->}[r]^\Delta &H^{m-1}(A\cup B) \ar@{->}[d] \ar@{->}[r]^-{(\iota_1^*,\iota_2^*)} &H^{m-1}(A_1\cup B)\oplus H^{m-1}(A_2\cup B) \ar@{->}[d] \\
				\tilde{H}^{m-2}(D) \ar@{->}[r]^\Delta &H^{m-1}(A) \ar@{->}[r] &H^{m-1}(A_1)\oplus H^{m-1}(A_2). }
			\end{align}
			If \( D \) is empty, then we still have the right hand commuting square, and the map out of \( H^{m-1}(A) \) is an isomorphism, and in particular, injective. This proves the lemma.
		\end{proof}

		\begin{lem}\label{lem:15A}
			Suppose \( A = \cup_{r=0}^N A_r \), where each \( A_r \) is compact. Let \( D_r = A_0 \cap A_r, \quad 1 \le r \le N \) and suppose \( A_r \cap A_s = \emptyset \), \( 1 \le r < s \le N \). Let \( m\geq 2 \). For each \( 1\le r\le N \), suppose \( B_r\supset D_r \) is compact and that the homomorphism \( \iota(B_r,D_r)^*: \tilde{H}^{m-2}(B_r) \to \tilde{H}^{m-2}(D_r) \) is zero. Furthermore, suppose that the intersection \( A_r\cap B_s \) is empty for all \( 1\leq r<s\leq N \). Let
			\begin{align*}
				&C = A \cup_{r=1}^N B_r, \\
				&C_0 = A_0 \cup_{r=1}^N B_r, \,\, \mathrm{ and}\\
				&C_r = A_r \cup B_r, \, 1 \le r \le N.
			\end{align*}
			Then \[ (\iota(C,A)^*)^{-1} (H^{m-1}(A)\setminus \{0\}) \subset \cup_{r=0}^N (\iota(C, C_r )^*)^{-1} (H^{m-1} (C_r)\setminus \{0\}). \]
		\end{lem}

		\begin{proof} 
			For \( 0\leq k\leq N-1 \), let \( E_k=A_0\cup\cdots\cup A_k\cup B_{k+1}\cup\cdots\cup B_N \), and Let \( E_N=A \). For \( 1\leq k\leq N \), we may apply Lemma \ref{lem:13A} to the sets \( ``A_1"=A_k \), \( ``A_2"=A_0\cup\cdots\cup A_{k-1}\cup B_{k+1}\cup\cdots\cup B_N \), and \( ``B"=B_k \), since the assumption \( A_k\cap B_j=\emptyset \) for all \( 1\leq k<j\leq N \) guarantees that \( ``A_1"\cap ``A_2"=D_k \). The following inclusion therefore holds for all \( 1\leq k\leq N \):
			\begin{align*}
				(\iota(E_k\cup B_k, E_k)^*)^{-1}&(H^{m-1}(E_k)\setminus\{0\})\subset\\
				 (\iota(E_k\cup B_k, E_{k-1})^*)^{-1}(H^{m-1}(E_{k-1})\setminus \{0\})&\cup (\iota(E_k\cup B_k, C_k)^*)^{-1}(H^{m-1}(C_k)\setminus \{0\}).
			\end{align*}
			Taking the inverse image in \( H^{m-1}(C) \) of the above sets by the map \( \iota(C, E_k\cup B_k)^* \), this yields
			\begin{align*}
				(\iota(C, E_k)^*)^{-1}&(H^{m-1}(E_k)\setminus\{0\})\subset\\
				(\iota(C, E_{k-1})^*)^{-1}(H^{m-1}(E_{k-1})\setminus \{0\})&\cup (\iota(C, C_k)^*)^{-1}(H^{m-1}(C_k)\setminus \{0\}).
			\end{align*}
			The result follows from downward induction on \( k \) starting at \( k=N \), since \( E_N=A \) and \( E_0=C_0 \).
		\end{proof}

		\begin{lem}\label{lem:16A}
			Suppose \( A = \cup_{r=0}^N A_r \) where each \( A_r \) is compact. Let \( D_r = A_{r-1}\cap A_r,\quad 1 \le r \le N \) and suppose \( A_r \cap A_s = \emptyset \) if \( |r-s| > 1 \). Let\footnote{Note the strict inequality.} \( m>2 \). For each \( 1\le r\le N \), suppose \( B_r\supset D_r \) is compact and that the homomorphism \( \iota(B_r,D_r)^*: H^{m-2}(B_r) \to H^{m-2}(D_r) \) is zero. Let \( B_0 = B_{N+1} = \emptyset \), and suppose further that \( B_r \cap A_{r-1}=D_r \) and \( B_r \cap A_s=\emptyset \) for all \( 1\leq r \leq N \) and \( 0\leq s < r-1 \). Let
			\begin{align*}
				&C = A \cup_{r=1}^N B_r,\\
				&C_{-1}= \cup_{r=1}^N B_r, \,\, \mathrm{ and}\\
				&C_r = B_r \cup A_r \cup B_{r+1},\,0\le r\le N.
			\end{align*}
			Then \[ (\iota(C,A)^*)^{-1}(H^{m-1}(A)\setminus \{0\}) \subset \cup_{r=-1}^N(\iota(C,C_r)^*)^{-1}(H^{m-1}(C_r)\setminus \{0\}). \]

			Furthermore, if \( B_r\cap B_s=\emptyset \) for all \( r\neq s \), then \[ (\iota(C,A)^*)^{-1}(H^{m-1}(A)\setminus \{0\}) \subset \cup_{r=0}^N(\iota(C,C_r)^*)^{-1}(H^{m-1}(C_r)\setminus \{0\}). \]
		\end{lem}

		\begin{proof}
			For \( 0\leq k\leq N \), let \( E_k=A_0\cup\cdots\cup A_k\cup B_{k+1}\cup\cdots\cup B_{N+1} \). Let \( E_{-1}=C_{-1} \) and \( D_0=\emptyset \). For \( 1\leq k\leq N \), let us apply Lemma \ref{lem:13A} to the sets \( ``A_1"=A_k \), \( ``A_2"=A_0\cup\cdots\cup A_{k-1}\cup B_{k+1}\cup\cdots\cup B_{N+1} \), and \( ``B"=B_k\cup B_{k+1} \). For \( k=0 \), use \( ``A_1"=A_0 \), \( ``A_2"=C_{-1} \), and \( ``B"=B_1 \). We may do so, because our assumption on the intersections \( B_r\cap A_s \) imply \( ``A_1"\cap ``A_2"=D_k\cup D_{k+1} \). This union being disjoint, the homomorphism \( \iota(``B",``D")^*=\iota(B_k\cup B_{k+1}, D_k\cup D_{k+1})^* \) is given by the direct sum \( \iota(B_k\cup B_{k+1}, D_k)^*\oplus \iota(B_k\cup B_{k+1},D_{k+1})^* \), both of which are zero. As in the proof of Lemma \ref{lem:15A}, the following inclusion therefore holds for all \( 0\leq k\leq N \):
			\begin{align*}
				(\iota(C, E_k)^*)^{-1}&(H^{m-1}(E_k)\setminus\{0\})\subset\\
				(\iota(C, E_{k-1})^*)^{-1}(H^{m-1}(E_{k-1})\setminus \{0\})&\cup (\iota(C, C_k)^*)^{-1}(H^{m-1}(C_k)\setminus \{0\}).
			\end{align*}
			This gives the first conclusion. If \( B_r\cap B_s=\emptyset \) for all \( r\neq s \), then by additivity, \[ (\iota(C, C_{-1})^*)^{-1}(H^{m-1}(C_{-1})\setminus \{0\})=\cup_r (\iota(C,B_r)^*)^{-1}(H^{m-1}(B_r)\setminus \{0\}).\] For each \( r \), we have \( (\iota(C,B_r)^*)^{-1}(H^{m-1}(B_r)\setminus \{0\})\subset (\iota(C,C_r)^*)^{-1}(H^{m-1}(C_r)\setminus \{0\}) \) by functoriality, thus giving the second conclusion.
		\end{proof}

		\begin{lem}\label{lem:17Apre}
			If the topological dimension \( \dim(A) \) of \( A \) is \( \le m-2 \), then \( H^{m-1}(A) = 0 \).
		\end{lem}

		\begin{proof}
			if \( m=1 \), the result is trivial. if \( m>1 \), then every open cover \( U \) admits a refinement \( V \) of order \( \le m-2 \) (\cite{hurewicz} Theorem V 1.) The nerve \( N \) of \( V \) is then a simplicial complex of dimension \( \le m-2 \), and so if \( x\in H^{m-1}(A) \) is represented by a simplicial cochain on the nerve of \( U \), it must pull back to the zero cochain on \( N \). Thus, \( x=0 \).
		\end{proof}

		\begin{lem}\label{lem:17A}
			 If \( \H^{m-1}(A) = 0 \), then \( H^{m-1}(A) = 0 \). 
		\end{lem}

		\begin{proof}
			If \( m=1 \), the result is trivial. If \( m>1 \), then \( \dim(A) \le m-2 \) by \cite{hurewicz} Theorem VII 3 and we may apply \ref{lem:17Apre}.
		\end{proof}

		\begin{lem}\label{lem:21B}
			Suppose \( X \) is compact and \( X=\varprojlim X_i \), where \( \{X_i\} \) is a system of compact surfaces with coboundary \( \supset L \), directed under inclusion. Then \( X \) is a surface with coboundary \( \supset L \).
		\end{lem}

		\begin{proof}
			By continuity of \v{C}ech cohomology, the obvious map \( \varinjlim \tilde{H}^{m-1}(X_i)\to \tilde{H}^{m-1}(X) \) is an isomorphism, and in particular a surjection, so the image of \( \iota(X,A)^* \) is the union of the images of \( \iota(X_i,A)^* \).
		\end{proof}

		\begin{lem}\label{lem:21C}
			If \( \{X_i\} \) is a sequence of compact surfaces with coboundary \( \supset L \) and \( X_i \to X \) in the Hausdorff metric, then \( X \) is a surface with coboundary \( \supset L \).
		\end{lem}

		\begin{proof}
			By Lemma \ref{lem:7A}, the sets \( Y_j = X \cup \cup_{i=j}^\i X_i \) satisfy the conditions of Lemma \ref{lem:21B}.
		\end{proof}

		\begin{lem}\label{lem:corespans}
			Suppose \( (X,A) \) is compact and \( X \) is a surface with coboundary \( \supset L \) such that \( \H^m(X\setminus A)<\infty \). If \( (Y,A)\subset (X,A) \) is compact, and \( \dim(X\setminus Y)\leq m-1 \), then \( Y \) is a surface with coboundary \( \supset L \). 
		\end{lem}

		\begin{proof}
			The inclusion of \( A \) into \( X \) factors through \( Y \), so it suffices to show \( \iota(X,Y)^*:\tilde{H}^{m-1}(X)\to \tilde{H}^{m-1}(Y) \) is surjective. For \( \e>0 \), let \( X_\e=X\cap \cN(Y, \e) \). Since \( X\setminus X_\e \) is compact and contained in \( X\setminus Y \), we may cover \( X\setminus X_\e \) by a finite number of open subsets \( U_i \) of \( X \), \( i=1,\dots,n \), such that for each \( i \), \( U_i\subset \cN(p_i,\e/2) \) for some \( p_i\in X\setminus X_\e \), and \( \dim(\p U_i)\leq m-2 \). Define \( B_\e=\overline{\cup_{i=1}^N U_i} \) and \( C_\e=X\setminus (\cup_{i=1}^N U_i) \). Then \( B_\e \) and \( C_\e \) are compact, \( B_\e\subset X\setminus Y \), and \( Y\subset C_\e\subset X_\e \). Furthermore, since \( B_\e\cap C_\e=\p(B_\e)\subset \cup{i=1}^N \p U_i \), it follows from \cite{hurewicz} Theorem III 1 that \( \dim(B_\e\cap C_\e)\leq m-2 \). By Lemma \ref{lem:17Apre}, \( H^{m-1}(B_\e\cap C_\e)=0 \). The Mayer-Vietoris sequence applied to the compact triad \( (X, B_\e, C_\e) \) thus implies that \( \iota(X,C_\e)^*:H^{m-1}(X)\to H^{m-1}(C_\e) \) is surjective. Finally, since \( Y=\varprojlim C_\e \), the result follows from the continuity of \v{C}ech cohomology.
		\end{proof}
		
		In particular,
		
		\begin{cor}\label{cor:corespans}
			If \( (X,A) \) is compact and \( X \) is a surface with coboundary \( \supset L \) such that \( \H^m(X\setminus A)<\infty \), then \( (X\setminus A)^*\cup A \) is a surface with coboundary \( \supset L \),
		\end{cor}
		
		\begin{proof}
			Let \( Y=(X\setminus A)^*\cup A \). Since \( \H^m(X\setminus Y)=0 \), it follows from \cite{hurewicz} Theorem VII 3 that \( \dim(X\setminus Y) \leq m-1 \).
		\end{proof}
		
		and
		
		\begin{cor}\label{cor:bddbelow0}
			If \( (X,A) \) is compact and \( X \) is a surface with coboundary \( \supset L\neq \emptyset \), then \( \H^m(X\setminus A)>0 \).
		\end{cor}
		\begin{proof}
			Suppose the result is false, and let \( Y=A \). Lemma \ref{lem:corespans} can be applied by \cite{hurewicz} Theorem VII 3. Lemma \ref{lem:1A} yields a contradiction.
		\end{proof}
		
		In Corollary \ref{cor:bddbelow}, we prove that \( \H^m(X\setminus A) \) cannot be arbitrarily small. 
		
	\subsection{Results specific to \( L^\mathcal{R} \)}\label{ssection:LR}

		\begin{thm}
			\label{thm:manifoldspans}
			Suppose \( A \) is an \( (m-1) \)-dimensional closed \( \mathcal{R} \)-orientable manifold and \( X \) is a compact \( \mathcal{R} \)-orientable manifold with boundary \( A \), then \( X \) is a surface with coboundary \( \supset L^\mathcal{R} \).
		\end{thm}

		\begin{proof}
			The result is obvious if \( m=1 \). Let \( m>1 \). Since \( A \) and \( X \) have the homotopy type of \( CW \)-complexes (see e.g. \cite{hatcher} Cor A.12,) we may treat the \v{C}ech cohomology groups involved in the definition of ``surface with coboundary'' as singular cohomology groups. We proceed by contradiction. Suppose there exists \( \phi\in L^\mathcal{R} \) with \( \iota(X,A)^*(\omega)=\phi \) for some \( \omega\in H^{m-1}(X;\cal{R}) \). Writing \( A=\cup A_i \) where the \( A_i \)'s are the connected components of \( A \), there exists \( j \) such that \( \iota(X,A_j)^*\omega \) is a generator of \( H^{m-1}(A_j;\cal{R}) \), and \( \iota(X,A_i)^*\omega=0 \) for all \( i\neq j \). 

			Let \( \eta\in H_m(X,A;\cal{R}) \) be a fundamental class for \( X \). Then \( \nu=\p\eta\in H_{m-1}(A;\cal{R}) \) is a fundamental class for \( A \) and by exactness, \( \iota(X,A)_*\nu=0 \). Write \( \nu=\sum \iota(A,A_i)_* \nu_i \). Then \( \nu_i \) is a fundamental class for \( A_i \) for each \( i \). We have
			\begin{align*}
				0&=\o(\iota(X,A)_*\nu)\\
				&=\iota(X,A)^*\o(\nu)\\
				&=\phi(\nu)\\
				&=\sum\phi\left(\iota(A,A_i)_*\nu_i\right)\\
				&=\sum \iota(A,A_i)^*\phi (\nu_i)\\
				&=\iota(A,A_j)^*\phi (\nu_j)\\
				&\neq 0
			\end{align*}
			where the last line follows from Poincar\'e duality.  
		\end{proof}

		By reducing mod 2, the universal coefficient theorem gives the following corollary:

		\begin{cor}
			\label{cor:nonorientablespans}
			If \( A \) is an \( (m-1) \)-dimensional closed orientable manifold and \( X \) is a compact manifold with boundary \( A \), then \( X \) is a surface with coboundary \( \supset L^\Z \).
		\end{cor}

		\begin{proof}
			It follows from Theorem \ref{thm:manifoldspans} that \( X \) is a surface with coboundary \( \supset L^{\Z/2\Z} \). Suppose there exists \( \omega\in H^{m-1}(X;\Z) \) and \( j \) such that \( \iota(X,A_j)^*\omega \) is a generator of \( H^{m-1}(A_j;\Z) \), and \( \iota(X,A_i)^*\omega=0 \) for all \( i\neq j \). The cohomology class \( \omega \) gives a homomorphism \( f_\o: H_{m-1}(X;\Z)\to \Z \), and by composing with the reduction map \( \Z \to \Z/2\Z \), a homomorphism \( \tilde{f}_\o:H_{m-1}(X;\Z)\to \Z/2\Z \). Since \[ H^{m-1}(X;\Z/2\Z)\to \mathrm{Hom}(H_{m-1}(X;\Z),\Z/2\Z)\to 0 \] is exact, the map \( \tilde{f}_\o \) lifts to a cohomology class \( \tilde{\o}\in H^{m-1}(X;\Z/2\Z) \). Since \( [A_i]_{\Z/2\Z} \), the \( \Z/2\Z \) fundamental class of \( A_i \), is the reduction mod-\( 2 \) of the \( \Z \) fundamental class of \( A_i \), denoted by  \( [A_i]_\Z \), it follows from naturality of the universal coefficient theorem and the tensor-hom adjunction that
			\begin{align*}
				\iota(X,A_i)^*\tilde{\o}\left([A_i]_{\Z/2\Z}\right)&=\iota(X,A_i)^*\tilde{f}_\o\left([A_i]_\Z\right)\\
				&=\tilde{f}_\o\left(\iota(X,A_i)_*[A_i]_\Z\right)\\
				&=\o\left(\iota(X,A_i)_*[A_i]_\Z\right)\textrm{ mod }2,
			\end{align*}
			which equals \( 1 \) if \( i=j \) and \( 0 \) otherwise. In other words, \( \iota(X,A)^*\tilde{\o}\in L^{\Z/2\Z} \), giving a contradiction.
		\end{proof}

		\begin{thm}
			\label{thm:flatspans}
			Suppose \( A \) is an \( (m-1) \)-dimensional closed orientable manifold. Suppose \( X \) is a compact set which can be written in the form \( X=A\cup\cup_i X_i \), where \( \cup_i X_i \) is a increasing union of a sequence \( \{X_i\} \) of compact manifolds with boundary, such that for each \( i \), \( \p X_i\cup A \) is the boundary of a compact manifold with boundary \( B_i \), and such that \( B_i\to A \) in the Hausdorff metric. Then \( X \) is a surface with coboundary \( \supset L^\Z \).
		\end{thm}

		\begin{proof}
			Writing \( C_N=X\cup \cup_{i=N}^\infty B_i \), it suffices to show, by Lemma \ref{lem:21B}, that for all \( N \), the set \( C_N \) is a surface with coboundary \( \supset L^\Z \). Since \( A\subset X_N\cup B_N\subset C_N \), it suffices to show, by Lemma \ref{lem:7A}, that \( X_N\cup B_N \) is a surface with coboundary \( \supset L^\Z \). Indeed it is, since the compact manifold formed by gluing \( X_N \) and \( B_N \) along their common boundary \( \p X_N \) is a manifold with boundary \( A \), and is thus a surface with coboundary \( \supset L^\Z \) by Corollary \ref{cor:nonorientablespans}. The set \( X_N\cup B_N \) is the continuous image of this manifold, and therefore is a surface with coboundary \( \supset L^\Z \) by Lemma \ref{lem:6A}.
		\end{proof}

		\begin{thm}
			\label{prop:unionspans}	
			Suppose \( A=A_1\cup\cdots\cup A_k \), where \( A_1,\dots, A_k \) are \( (m-1) \)-dimensional closed \( \mathcal{R} \)-orientable manifolds. Suppose that every non-empty intersection of the \( A_i \)'s is also a \( (m-1) \)-dimensional closed manifold, or equivalently that every component of \( A \) is contained in some \( A_i \). Then \( A \) is a \( \mathcal{R} \)-orientable closed manifold, and if \( X_i \) is a surface with coboundary \( \supset L^\mathcal{R}(A_i) \), \( i=1,\dots k \), then \( X=\cup_i X_i \) is a surface with coboundary \( \supset L^\mathcal{R}(A) \).
		\end{thm}

		\begin{proof}
			The equivalence of the assumptions in the second sentence is a consequence of Brouwer's invariance of domain theorem \cite{brouwer}. Then \( A \), being the disjoint union of its connected components, is a \( \mathcal{R} \)-orientable closed manifold. Moreover, every component of \( A_i \) is a component of \( A \), and every component of \( A \) is a component of \( A_i \) for some \( i \). The result follows.
		\end{proof}
	
\section{Hausdorff spherical measure}
\label{sec:hausdorff_spherical_measure}
	
	Lemma \ref{lem:4} is a sharp version of the Eilenberg inequality \cite{eilenberg} for Hausdorff spherical measure, and is a generalization of \cite{reifenberg} Lemma 4.
	
	\begin{lem}[Slicing inequality]
		\label{lem:4}
		Suppose \( X\subset \R^n \) is \( \sp^m \) measurable with \( \sp^m(X)<\i \), and \( f:X \to \R \) is \( \sp^m \) measurable, non-negative and bounded. If \( X_t \) denotes the set of points of \( X \) at distance \( t \) from a fixed \( M \)-dimensional affine subspace \( E \) of \( \R^n \), where \( 0 \leq M \leq n-1 \), then \[ \int_0^\i \int_{X_t}f(q) d\sp^{m-1}(q) dt \le \int_X f(q) d\sp^m(q). \]
	\end{lem}

	\begin{proof}
		We first show that \( f\lfloor_{X_t} \) is \( \sp^{m-1} \) measurable for almost every \( t \) and that \( g(t) := \int_{X_t}f(q)d\sp^{m-1}(q) \) is measurable. If \( f \) is a simple function, this follows from \cite{federer} 2.10.26 (the statement is for Hausdorff measure, but the proof equally applies to spherical measure.) Since \( f \) is a pointwise limit of an increasing sequence \( \{f_i\} \) of simple functions, \( f\lfloor_{X_t} \) is measurable for almost every \( t \) (namely, the points \( t \) for which \( f_i\lfloor_{X_t} \) is measurable for all \( i \).)  It follows from the monotone convergence theorem that \( g \) is the pointwise supremum of a sequence of measurable functions and hence is measurable.

		Let \( \e > 0 \) and produce from Lusin's theorem a closed subset \( Y\subset X \) such that \( \sp^m(X\setminus Y)<\e \) and \( f\lfloor_Y \) is continuous. For each \( y\in Y \), let \( \d_y \) be small enough so that \( |f(z)-f(y)|<\e \) for all \( z\in B(y,\d_y) \). Let \( \cal{V}_\e = \{B(p_i,r_i)\}_{i\in I} \) be a covering of \( \sp^m \) almost all \( Y \) by disjoint balls of radius \( r_i < \d_{p_i} \). If \( V_\e \) denotes the union of the balls in \( \cal{V}_\e \) and \( \kappa<\i \) is an upper bound for \( f \), then by \cite{reifenberg} Lemma 4,
		\begin{align*}
			\int_0^\i \int_{X_t}f(q)\, d\sp^{m-1}(q) \,dt &= \sum_{i\in I}\int_0^\i \int_{X(p_i,r_i)_t} f(q)\, d\sp^{m-1}(q) \,dt + \int_0^\i \int_{(X\setminus V_\e)_t} f(q)\, d\sp^{m-1}(q) \,dt\\
			&\leq \sum_{i\in I}(f(p_i)+ \e) \int_0^\i \sp^{m-1}(X(p_i,r_i)_t) \,dt + \kappa \int_0^\i \sp^{m-1}((X\setminus V_\e)_t) \,dt\\
			&\leq \sum_{i\in I}(f(p_i)+ \e) \sp^m(X(p_i,r_i)) + \kappa \e \\
			&\leq \sum_{i\in I} \int_{X(p_i,r_i)} (f(q) + 2\e)\, d\sp^m(q) + \kappa \e\\
			&\leq \int_X f(q) d\sp^m(q) + 2\e \sp^m(X) + \kappa \e.
		\end{align*}
	 \end{proof}

	One can replace the sets \( X_h \) in Lemma \ref{lem:4} with the level sets of a suitably regular function. It is enough that the level sets smoothly foliate \( \R^n \) except possibly a set of zero \( (m-1) \)-dimensional Hausdorff measure.

	The following is a cohomological version of \cite{reifenberg} Lemma 12.
	\begin{lem}
		\label{lem:12}
		Let \( p \in \R^n \) and let \( E \) be an affine \( m \)-plane containing \( p \). Fix \( 0 < \e < 1/2 \), \( r>0 \), and suppose \( A \subset (\fr\,B(p,r))\cap \cal{N}(E,\e r) \) is compact. If \( L\subset \tilde{H}^{m-1}(A)\setminus \{0\} \) and \( X \) is a surface with coboundary \( \supset L \), then there exists a compact set \( Y\subset (\fr\, B(p,r)) \cap \cal{N}(E,\e r) \) such that
		\begin{equation}
			\label{eq:12B}
			\sp^m(Y) \leq \e r\,\frac{2^{2m}\a_m}{\a_{m-1}}\,\sp^{m-1}(A)
		\end{equation}
		and either 
		\begin{equation}
			\label{eq:12A}
			\sp^m(X\cup Y) \geq \a_m r^m
		\end{equation}
		and the orthogonal projection of \( X \) onto \( E \) contains \( E\cap B(p,(1-\e)r) \), or \( Y \) is a surface with coboundary \( \supset L \). Moreover, if \( A \) is \( (m-1) \)-rectifiable, then \( Y \) is \( m \)-rectifiable.
	\end{lem}
	
	\begin{proof}
		Let \( A' \) denote the radial projection from \( p \) onto \( \fr\,B(p,r) \) of the orthogonal projection of \( A \) onto \( E \). Let \( B=A\cup A' \) and let \( Y \) be the radial projection onto \( \fr \, B(p,r) \) of the set \( \cup_{x\in A} I_x, \) where \( I_x \) denotes the line segment joining \( x \) to its orthogonal projection in \( E \). It is straightforward to see that \( Y \) is \( m \)-rectifiable whenever \( A \) is \( (m-1) \)-rectifiable.
			
		By Lemma 6 of \cite{reifenberg},
		\begin{equation}
			\label{eq:12(ii)}
			\sp^m(Y) \le \e r \, \frac{2^m\a_m}{\a_{m-1}}\, \sp^{m-1}(A)\cdot\left(\frac{r}{(1-\e)r}\right)^m\leq \e r\,\frac{2^{2m}\a_m}{\a_{m-1}}\,\sp^{m-1}(A).
		\end{equation}

		By Lemma \ref{lem:10A} there exists a subset \( L' \subset \tilde{H}^{m-1}(A')\setminus \{0\} \) such that
		\begin{equation}
			\label{eq:12(i)}
			K^*(Y,B) \cup (\iota(B,A)^*)^{-1}(L) = K^*(Y,B) \cup (\iota(B,A')^*)^{-1}(L').
		\end{equation}

		If \( L' = \emptyset \), then \eqref{eq:12(i)} gives \( (\iota(B,A)^*)^{-1}(L) \subset K^*(Y,B) \). It follows from Lemma \ref{lem:11A} that \( Y \) is a surface with coboundary \( \supset L \).
			
		Now suppose \( L'\ne \emptyset \). It suffices to show that the orthogonal projection of \( X\cup Y \) onto \( E \) contains \( E\cap B(p,r) \). We know by Alexander duality that \( A'=(\fr\,B(p,r))\cap E \) (otherwise \( \tilde{H}^{m-1}(A')=0 \), contradicting \( L'\ne\emptyset \).) Therefore, by Lemmas \ref{lem:8A} and \ref{lem:6A}, it suffices to show that \( X\cup Y \) is a surface with coboundary \( \supset L' \). This follows from Lemma \ref{lem:11A} and \eqref{eq:12(i)}.

			\begin{figure}[htbp]
						\centering
						\includegraphics[height=2in]{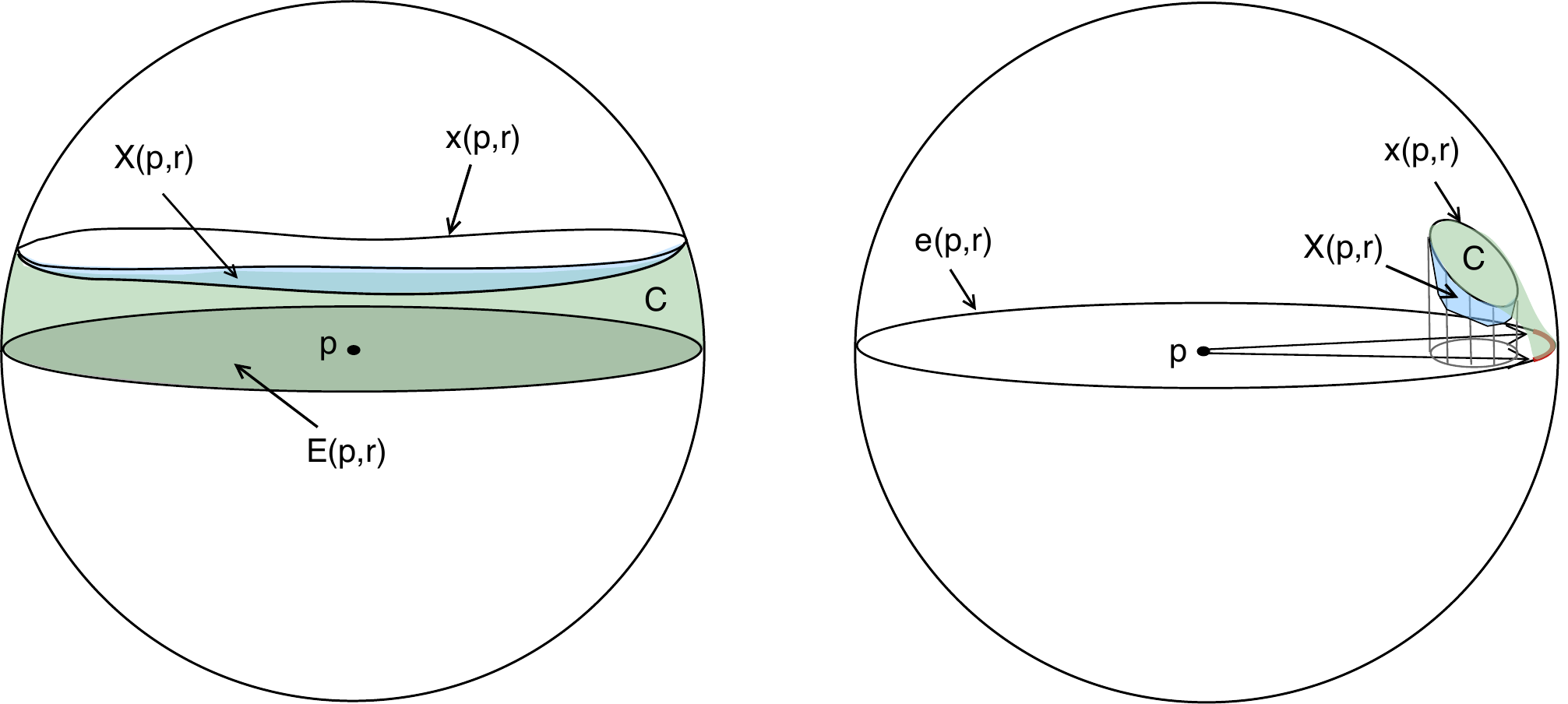}
						\caption{Two cases for Lemma \ref{lem:12}}
						\label{fig:Lemma12}
					\end{figure}
	\end{proof}

\section{Isoperimetry}
\label{section:isoperimetry}

	The goal of this section is to prove that if \( U \) is an open set disjoint from \( A \), and \( C\subset U \) is compact, then there exists a constant \( K_{C,U}>0 \) such that if \( X \) is a compact surface with coboundary \( \supset L \) such that \( \H^m(X\cap U)<K_{C,U} \), then there exists a compact surface \( Y \) with coboundary \( \supset L \) disjoint from \( C \). As a corollary, this will imply that if \( L\neq \emptyset \), then the \( m \)-dimensional Hausdorff measure of a compact surface with coboundary \( \supset L \) cannot be arbitrarily small.

	\begin{lem}
		\label{lem:8}
		Suppose \( m\geq 2 \). There exists a constant \( 0<K_m^n <\infty \) such that if \( A \) is compact, then there exists a compact surface \( X \) with coboundary \( \supset H^{m-1}(A) \setminus \{0\} \) such that
		\begin{enumerate}
			\item\label{lem:8:item:2} \( X \) is contained in the convex hull of \( A \);
			\item\label{lem:8:item:3} \( X \subset \cN(A, K_m^n \,(\sp^{m-1}(A))^{1/(m-1)}) \);
			\item\label{lem:8:item:4} \( \sp^m(X) \le K_m^n \,(\sp^{m-1}(A))^{m/(m-1)} \); and
			\item\label{lem:8:item:5} If \( A \) is \( (m-1) \)-rectifiable, then \( X \) is \( m \)-rectifiable.
		\end{enumerate}
	\end{lem}

	\begin{proof}
		\eqref{lem:8:item:2}-\eqref{lem:8:item:4} is a cohomological version of Lemma 8 of \cite{reifenberg}, and the proof is identical for \( m \ge 2 \), replacing Lemmas 3A, 11A, and 16A with Lemmas \ref{lem:3A}, \ref{lem:11A}, and \ref{lem:16A}, respectively\footnote{The eagle-eyed reader will see that we added a few hypotheses to the assumptions in Lemma \ref{lem:16A}. A proof of the general statement was not apparent to us, and Adams did not provide a proof for his homological version. However, Lemma \ref{lem:16A}, with its additional assumptions, remains applicable in the proof of Lemma \ref{lem:8}.}. \eqref{lem:8:item:5} is a straightforward generalization.
	\end{proof}
	
	For notational purposes, let \( K_1^n=1 \).
	\begin{lem}
		\label{lem:haircut}
		Suppose \( A \) is compact, \( X \) is a compact surface with coboundary \( \supset L \) and \( \cN(p,r) \) is disjoint from \( A \), with \[ \sp^m(X(p,r)) \le \frac{r^m}{2\left(2^M m\right)^m\left(2^{n+M-1} K_m^n\right)^{m-1}}, \] where \( M\geq 1 \). If \( W \subset (r/2^M,r) \) has full Lebesgue measure, then there exist \( r'\in W \) and a compact surface \( \hat{X} \) with coboundary \( \supset L \) such that
	 	\begin{enumerate}
	 		\item\label{lem:haircut:item:0} \( \hat{X} \cap (\cN(p,r')^c) = X \cap (\cN(p,r')^c) \),
			\item\label{lem:haircut:item:1} \( \hat{X}(p,r') \) is contained in the convex hull of \( x(p,r') \),
			\item\label{lem:haircut:item:2} \( \hat{X}(p,r') \subset B(p,r')\setminus B(p,r'/2) \), and
	 		\item\label{lem:haircut:item:3} \( \sp^m(\hat{X}(p,r')) \le \frac{1}{2^{M+m-1}}\sp^m(X(p,r')). \)
	 	\end{enumerate}
		Moreover, if \( x(p,r') \) is \( (m-1) \)-rectifiable, then \( \hat{X}(p,r') \) is \( m \)-rectifiable.
	 \end{lem}

	\begin{proof}
		The \( m=1 \) case is immediate, since Lemma \ref{lem:4} implies that there exists \( r'\in W \) such that \( x(p,r') \) is empty. The set \( \hat{X}=X\setminus B(p,r') \) satisfies the desired requirements by Lemma \ref{lem:13}.
		
		Now assume \( m>1 \). We first prove that if \( W \subset (r/2^M,r) \) has full Lebesgue measure, then there exists \( r' \in W \) such that
		\begin{align}\label{haircutinequality}
			\sp^{m-1}(x(p,r'))^{m/(m-1)} \le \frac{1}{2^{n+M-1} K_m^n} \int_0^{r'} \sp^{m-1} (x(p,t))dt.
		\end{align}
 		Suppose there is no such \( r' \). Let \( I(p,s):=\int_0^s \sp^{m-1}(x(p,t))dt \). By the Lebesgue Differentiation Theorem, \[ \frac{\frac{d}{ds}I(p,s)}{I(p,s)^{(m-1)/m}} > \frac{1}{(2^{n+M-1} K_m^n)^{(m-1)/m}} \] for almost every \( r/2^M < s < r \). Integrating, this implies
		\begin{align} 
			\label{lem:halfradius:eq:1p}
			\int_{r/2^M}^r \frac{d}{ds}\left(I(p,s)^{1/m}\right) ds > \frac{r}{2^M m\left(2^{n+M-1} K_m^n\right)^{(m-1)/m}}.
		\end{align}  
		Since \( I(p,s) \) is increasing and absolutely continuous, the function \( I(p,s)^{1/m} \) is also absolutely continuous, and so the left hand side of \eqref{lem:halfradius:eq:1p} equals  \( I(p,r)^{1/m} - I(p,r/2^M)^{1/m} \). Lemma \ref{lem:4} gives \[ \sp^m(X(p,r)) \ge I(p, r) > \frac{r^m}{\left(2^M m\right)^m\left(2^{n+M-1} K_m^n\right)^{m-1}}, \] contradicting our initial assumption. 

		Lemma \ref{lem:4} and \eqref{haircutinequality} yield the existence of some \( r' \in W \) such that
		\begin{align}
			\label{lem:haircut:eq:1}
			\sp^{m-1}(x(p,r'))^{m/(m-1)} \leq \frac{1}{2^{n+M-1} K_m^n}\sp^m(X(p,r')).
		\end{align}
		Thus, our initial assumption on \( \sp^m(X(p,r)) \) implies 
		\begin{align*}
			\sp^{m-1}(x(p,r'))^{1/(m-1)} \leq \frac{1}{2^{(n+M-1)/m} (K_m^n)^{1/m}}\sp^m(X(p,r))^{1/m} < \frac{r'}{2K_m^n},
		\end{align*}
		and we may apply Lemmas \ref{lem:8} and \ref{lem:13} to find the required set \( \hat{X} \).
	\end{proof}
 
	\begin{cor}
		\label{cor:gridsmall}
		Suppose \( A \) is compact, \( X \) is a compact surface with coboundary \( \supset L \), and \( Q \) is a closed \( n \)-cube of side length \( \ell \) whose interior is disjoint from \( A \), with \[ \sp^m(X \cap \mathring{Q}) < \frac{\ell^m}{2\left(4m\right)^m\left(2^n K_m^n\right)^{m-1}}. \] Then there exist a compact surface \( X' \) with coboundary \( \supset L \) and \( Y\subset \fr\, Q \) such that
		\begin{enumerate}
			\item \( X'\cap Q^c = X\cap Q^c \),
			\item \( X'\cap Q = (X\cap \fr\, Q)\cup Y \), and
			\item \( \sp^m(Y)\leq (4n)^m \sp^m(X\cap \mathring{Q}). \)
		\end{enumerate}
		Moreover, if \( X\cap Q \) is \( m \)-rectifiable, then \( X'\cap Q \) is \( m \)-rectifiable.
	\end{cor}

	\begin{proof}
		Let \( p \) be the center point of \( Q \). Then \( B(p,\ell/2) \subset Q \) satisfies the conditions of Lemma \ref{lem:haircut}, so there exists \( \ell/4 < r < \ell/2 \) and a compact surface \( \hat{X} \) with coboundary \( \supset L \) satisfying \ref{lem:haircut} \eqref{lem:haircut:item:0}-\eqref{lem:haircut:item:3}.

		Thus, \( \hat{X}\cap \mathring{Q} \) is contained in the region \( \mathring{Q} \cap \cN(p,r/2)^c \), so we may project \( \hat{X}\cap \mathring{Q} \) from \( p \) to a set \( Y\subset \fr\, Q \), and the image \( X' \) of \( \hat{X} \) under this projection is a surface with coboundary \( \supset L \) by Lemma \ref{lem:6A}. Moreover, since the Lipschitz constant of the projection is bounded above by \( 4n \), we have
		\begin{align*}
			\sp^m(Y) &\le (4n)^m\sp^m(\hat{X} \cap \mathring{Q})\\
			&= (4n)^m\left[\sp^m(X \cap \mathring{Q}\cap (B(p,r)^c))+\sp^m(\hat{X}(p,r))\right]\\
			&\leq (4n)^m\left[\sp^m(X \cap \mathring{Q}\cap (B(p,r)^c))+ \frac{1}{2^m}\sp^m(X(p,r))\right]\\
			&\leq (4n)^m \sp^m(X\cap \mathring{Q}).
		\end{align*}
	\end{proof}
	
	\begin{defn}
		\label{def:binarysubdivision}
		A collection \( \cal{Q} \) of closed \( n \)-cubes is a \emph{\textbf{dyadic subdivision of \( \R^n \)}} if \( \cal{Q}=\sqcup_{k\in \Z}\cal{Q}_k \), where each \( \cal{Q}_k \) is a cover of \( \R^n \) by \( n \)-cubes of side length \( 2^{-k} \) that intersect only on faces, and such that \( \cal{Q}_{k+1} \) is a refinement of \( \cal{Q}_k \). The \textbf{\emph{\( k \)-skeleton}} of a sub-collection \( \cal{Q}' \) of \( \cal{Q} \) is the union of the \( k \)-dimensional faces of the cubes in \( \cal{Q}' \).
	\end{defn}

	\begin{thm}
		\label{thm:poof}
		Suppose \( A \) is compact. Given \( U\subset \R^n \) an open set disjoint from \( A \), a compact set \( C\subset U \) and \( \e>0 \), there exist constants \( 0<\cal{K}_{C,U,\e}<\i \) and \( 0<\cal{L}_{C,U,\e}<\i \), independent of \( A \), such that if \( X \) is a compact surface with coboundary \( \supset L \) satisfying \( \H^m(X\cap U)<\cal{K}_{C,U,\e} \), then there exists a compact surface \( Y \) with coboundary \( \supset L \) disjoint from \( C \), such that 
		\begin{enumerate}
			\item \( Y\subset \cN(X,\e) \),
			\item \( \H^m(Y\cap U)\leq \cal{L}_{C,U,\e} \H^m(X\cap U) \),
			\item \( Y\cap (U^c)=X\cap (U^c) \),
			\item If \( \H^m(X\setminus A)<\i \) and \( X=(X\setminus A)^*\cup A \), then \( Y=(Y\setminus A)^*\cup A \),
			\item If \( X\setminus A \) is \( m \)-rectifiable, then \( Y\setminus A \) is \( m \)-rectifiable.
		\end{enumerate} 
	\end{thm}
	
	The idea of the proof is to cover \( C \) with a grid of cubes, then repeatedly apply Corollary \ref{cor:gridsmall} to push \( X \) into lower dimensional skeleta. Once \( X \) is contained in the \( m \)-skeleton, we simply project \( X \) onto the \( (m-1) \)-skeleton, after which point Lemma \ref{lem:corespans} can be applied. We will need \( \H^m(X\cap U) \) to be small enough so that we may apply Corollary \ref{cor:gridsmall} successively, and so that once \( X \) has been pushed onto the \( m \) skeleton, there is not enough material to cover even one \( m \)-face of a single cube, so that we can perform the projection.
	
	\begin{proof}
		Let \( \cal{Q} \) be a dyadic subdivision of \( \R^n \) and let \( k \) be large enough so that the diameter \( 2^{-k}\sqrt{n} \) of a cube in \( \cal{Q}_k \) is less than \( \e \), and so that there exists a collection \( \varphi=\{Q_1,\dots,Q_N\}\subset \cal{Q}_k \) of cubes contained in \( U \), such that \( C \) is contained in the interior of \( \Phi\equiv \cup_{Q\in \varphi}Q \).
		
		Suppose \( X\cap \mathring{\Phi} \) is contained in the \( m \)-skeleton of \( \varphi \) and that \( \sp^m(X\cap U)<2^{-km} \). Let \( \{F_1,\dots, F_S\} \) denote the \( m \)-faces of \( \varphi \) which are not contained in \( \fr\,\Phi \). For each \( 1\leq i\leq S \) there exists a point \( p_i\in F_i\setminus X \), and so if \( \pi_i \) denotes the radial projection from \( p_i \) onto the \( (m-1) \)-faces of \( F_i \), the set \( \hat{X}\equiv \pi_S\circ\cdots\circ \pi_1 X \) is still, by Lemma \ref{lem:6A}, a surface with coboundary \( \supset L \). But now, \( \hat{X}\cap \mathring{\Phi} \) is contained in the the \( (m-1) \)-skeleton of \( \varphi \), and so Lemmas \ref{lem:13} and \ref{lem:corespans} imply that \( \hat{X}\setminus \mathring{\Phi} \) is a surface with coboundary \( \supset L \).

		For \( m\leq j\leq n \), let \( P(j) \) be the following statement: \emph{There exist \( \cal{K}_j>0 \) and \( \cal{L}_j<\i \) such that if \( X \) is a compact surface with coboundary \( \supset L \) such that \( X\cap\mathring{\Phi} \) is contained in the \( j \)-skeleton of \( \varphi \), and such that \( \sp^m(X\cap U)<\cal{K}_j \), then there exists a compact surface \( Y \) with coboundary \( \supset L \) disjoint from \( C \) such that \( \sp^m(Y\cap U)<\cal{L}_j\sp^m(X\cap U) \) and \( Y\cap (U^c)=X\cap (U^c) \).} 
		
		We have proved \( P(m) \) (setting \( \cal{K}_m=2^{-km} \) and \( \cal{L}_m=1 \),) and we wish to prove \( P(n) \). Supposing we have proved \( P(j-1) \), let us prove \( P(j) \).
		
		Let \( \mathcal{G}=\{G_1,\dots, G_S \} \) denote the \( j \)-faces of \( \varphi \) which are not contained in \( \fr\,\Phi \). Let \( \cal{K}_j \) be the constant
		\[
		\inf\left\{\frac{\cal{K}_{j-1}}{(1+(4j)^m)^S},\,\, (1+(4j)^m)^{-S}2^{-km-1}\left(4m\right)^{-m}\left(2^j K_m^j\right)^{1-m}\right\},
		\]
		and suppose \( X \) satisfies the conditions of \( P(j) \). The set \( X\cap G_1\subset \R^j \) is a compact surface with coboundary \( \supset K^*(X\cap G_1, X\cap \p G_1) \), where \( \p G_i \) denotes the union of the \( (j-1) \)-faces of \( G_i \). Therefore by Corollary \ref{cor:gridsmall} and Lemma \ref{lem:13}, there exists a surface \( Z_1\subset \p G_1 \) with coboundary \( \supset K^*(X\cap G_1, X\cap \p G_1) \), such that
		\begin{enumerate}
			\item \( X_1\equiv (X\setminus (X\cap G_1))\cup Z_1 \) is a compact surface with coboundary \( \supset L \),
			\item \( X_1\cap \mathring{\Phi} \) is contained in the \( j \)-skeleton of \( \varphi \),
			\item \( X_1\cap G_1\subset \p G_1 \), and
			\item \( \sp^m(X_1\cap U)\leq (1+(4j)^m)\sp^m(X \cap U) \).
		\end{enumerate}

		Now let us repeat the above construction for the remaining \( j \)-faces in \( \mathcal{G} \): Let \( 1\leq s< S \) and suppose there exists a compact surface \( X_s \) with coboundary \( \supset L \) such that
		\begin{enumerate}
			\item \( X_s\cap \mathring{\Phi} \) is contained in the \( j \)-skeleton of \( \varphi \),
			\item \( X_s\cap G_r\subset \p G_r \) for all \( 1\leq r\leq s \), and
			\item \( \sp^m(X_s\cap U) \leq (1+(4j)^m)^s \sp^m(X \cap U) \).
		\end{enumerate}
		
		Then by Corollary \ref{cor:gridsmall} and Lemma \ref{lem:13}, there exists a compact surface \( X_{s+1} \) with coboundary \( \supset L \) such that 
		\begin{enumerate}
			\item \( X_{s+1}\cap \mathring{\Phi} \) is contained in the \( j \)-skeleton of \( \varphi \),
			\item \( X_{s+1}\cap G_r\subset \p G_r \) for all \( 1\leq r\leq s+1 \), and
			\item \( \sp^m(X_{s+1}\cap U) \leq (1+(4j)^m)^{s+1} \sp^m(X \cap U) \).
		\end{enumerate}
		The set \( X_S \) satisfies the conditions of \( P(j-1) \), and so setting \( \cal{L}_j=\cal{L}_{j-1}(1+(4j)^m)^S \), we have proved \( P(j) \).
		
		Since \( \H^m\leq \sp^m\leq 2^m \H^m \), we convert the spherical measure constants \( \cal{K}_n \) and \( \cal{L}_n \) into Hausdorff measure constants by letting \( \cal{K}_{C,U,\e}= 2^{-m} \cal{K}_n \) and \( \cal{L}_{C,U,\e} = 2^m\cal{L}_n \). Let \( Y \) be the set \( X_S \) achieved in statement \( P(n) \). Per our construction, since \( 2^{-k}\sqrt{n}<\e \), it holds that \( Y\subset \cN(X,\e) \) is a surface with coboundary \( \supset L \) disjoint from \( C \). The constants \( \cal{K}_{C,U,\e} \) and \( \cal{L}_{C,U,\e} \) depend only on \( n,m,C,U,\e,\cal{Q},k \) and \( \varphi \).
		
		Finally, if \( \H^m(X\setminus A)<\i \) and \( X=(X\setminus A)^*\cup A \) then replace \( Y \) with the set \( Y'\equiv (Y\setminus A)^*\cup A \), which is a surface with coboundary \( \supset L \) by Corollary \ref{cor:corespans}. Then \( Y'=(Y'\setminus A)^*\cup A \), and since \( Y'\subset Y \), it also holds that \( Y'\subset \cN(X,\e) \) and \( \H^m(Y'\cap U)< \cal{L}_{C,U,\e} \H^m(X\cap U) \). It remains to show that \( Y'\cap (U^c)=X\cap (U^c) \). But this holds too, since by way of the above construction, \( Y\cap (\Phi^c)=X\cap (\Phi^c) \), and \( \Phi\subset U \).
	\end{proof}
	
	\begin{defn}
		Let \( h \) denote the inclusion of \( \R^n \) into its one-point compactification \( S^n \). The set \( A\subset \R^n \) is compact, so the Alexander dual of \( l\in H^{m-1}(A); \Z) \) in \( H_{n-m}(S^n\setminus h(A)) \) is represented by a singular cycle \( \sigma=\sum n_i \sigma_i \), where each \( \sigma_i \) is a map from the \( (n-m) \)-simplex \( \D_{n-m} \) to \( S^n\setminus h(A) \), and \( n_i\in G\setminus \{0\} \). Let \( R=\cup_i \sigma_i(\D_{n-m}) \). Call \( h^{-1}(R) \) the \emph{\textbf{carrier in \( \R^n \)}} of \( \sigma \).
	\end{defn}

	\begin{cor}
		\label{cor:bddbelow}
		 If \( A \) is compact and \( L\neq \emptyset \), then there exists a constant \( \cal{K}_{L,A}>0 \) such that if \( X \) is a compact surface with coboundary \( \supset L \), then \( \H^m(X\setminus A)\geq \cal{K}_{L,A} \).
	\end{cor}
		
	\begin{proof}
		Let \( r<\i \) be large enough so that \( A\subset B(0,r) \). Let \( l\in L \) and let \( C \) be the intersection of \( B(0,2r) \) with the carrier in \( \R^n \) of some singular cycle \( \sigma \) representing the Alexander dual of \( l \). Then \( C \), being compact, has a neighborhood \( U \) disjoint from \( A \). Let \( \cal{K}_{L,A}=\cal{K}_{C,U,r} \) be the constant achieved in Theorem \ref{thm:poof}.
		
		Now suppose by contradiction that \( X \) is a compact surface with coboundary \( \supset L \), and that \( \H^m(X\setminus A)<\cal{K}_{L,A} \). Let \( \pi \) denote the radial projection from \( \R^n \) to \( B(0,r) \), fixing \( B(0,r) \). By Lemma \ref{lem:6A}, \( \pi(X) \) is a surface with coboundary \( \supset L \), and since \( \pi \) is a contraction, we have \( \H^m(\pi (X\setminus A))\leq \H^m(X\setminus A) \). Applying Theorem \ref{thm:poof} to \( \pi(X) \), we get a compact surface \( Y\subset B(0,2r) \) with coboundary \( \supset L \) disjoint from \( C \), and hence also disjoint from the carrier in \( \R^n \) of \( \sigma \). It follows that \( [\sigma]\in H_{n-m}(S^n\setminus h(A)) \) is in the image of the homomorphism \( H_{n-m}(S^n\setminus h(Y))\to H_{n-m}(S^n\setminus h(A)) \) induced by the inclusion \( S^n\setminus h(Y)\to S^n\setminus h(A) \), and hence \( l \) extends over \( Y \) by naturality of the Alexander duality isomorphism, a contradiction.
	\end{proof}
	
	As a special case of Corollary \ref{cor:bddbelow}, setting \( L=L^\Z \),
	\begin{cor}
		If \( A\subset \R^n \) is a compact \( (m-1) \)-dimensional manifold and \( X\supset A \) is a compact manifold with boundary \( A \), then \( \H^m(X) \) cannot be arbitrarily small if \( A \) is fixed.
	\end{cor}
	
	\begin{remarks}
		The proof of Corollary \ref{cor:bddbelow} shows something slightly stronger, that if \( X \) is any surface with coboundary \( \supset L \) such that \( \pi(X) \) is compact, then \( \H^m(X\setminus A)\geq \cal{K}_{L,A} \). It may be possible to remove the assumption on compactness altogether, but it is not clear to us how to do this. Theorem \ref{thm:poof} also holds for Reifenberg surfaces with boundary, but since our proof of Corollary \ref{cor:bddbelow} relies on Alexander duality, it only works for surfaces with coboundary. The corresponding statement for Reifenberg surfaces with boundary follows from Corollary \ref{cor:closedinweaktop} below (which also works for surfaces with boundary,) weak compactness of measures and Lemma \ref{lem:21B}.
	\end{remarks}
	
	\subsection{Convergence in the Hausdorff metric}
		\label{sub:haircuts}
		
		\begin{thm}
			\label{thm:haircut}
			Suppose \( (C,A) \) is compact and that \( C \) is a Lipschitz neighborhood retract. Suppose \( 0<\l\leq \k<\i \) and that for each \( Y\in \Span(A,C,L,m) \), a Borel measure \( \mu_Y \) on \( \R^n \) is given satisfying \[ \l \H^m\lfloor_{Y\setminus A} \leq \mu_Y\leq \k\H^m\lfloor_{Y\setminus A} \] and such that for every Borel set \( \cal{B}\subset \R^n \), it holds that \( \mu_Y(\cal{B})=\mu_Z(\cal{B}) \) whenever \( (Y\setminus A)\cap \cal{B}=(Z\setminus A)\cap \cal{B} \). If \( \{X_k\}\subset \Span(A,C,L,m) \) is a sequence such that \( \mu_{X_k} \) converges weakly to a finite Borel measure \( \mu \), then there exists a sequence \( \{X_k'\}\subset \Span(A,C,L,m) \) converging to \( X\equiv\supp(\mu)\cup A \) in the Hausdorff metric, such that \( \mu_{X_k'}\) converges weakly to \( \mu \) and such that \( X_k'=\supp(\mu_{X_k'})\cup A \) for each \( k \).
		\end{thm}
		
		The above theorem is also valid if we change the definition of \( \Span(A,C,L,m) \) so that for every \( X\in \Span(A,C,L,m) \), the set \( X\setminus A \) is \( m \)-rectifiable.
				
		\begin{proof}
			First, observe that if \( Y\in \Span(A,C,L,m) \), then \( \mu_Y= \mu_{(Y\setminus A)^*\cup A} \) and
			\begin{equation}
				\label{thm:haircut:aq0}
				(Y\setminus A)^*\cup A = \supp(\mu_Y)\cup A.
			\end{equation}
			Furthermore, if \( Z=(Y\setminus A)^*\cup A \), then 
			\begin{equation}
				\label{thm:haircut:aq1}
				Z=(Z\setminus A)^*\cup A.
			\end{equation}
			Thus, we may assume without loss of generality that \( X_k=\supp(\mu_{X_k})\cup A \) for each \( k \).
			
			Let \( \pi:U\to C \) be a Lipschitz retraction of an open neighborhood \( U \) of \( C \). Let \( \nu_k\to 0 \). For each \( k \) large enough so that \( \cN(C,\nu_k/(2+2\lip(\pi)))\subset U \), let \( C_k=B(C,\nu_k)\setminus \cN(X,\nu_k/2) \) and let \( U_k=\cN(C_k,\nu_k/4) \). Produce from Theorem \ref{thm:poof} the constants \( \cal{K}_{C_k,U_k,\nu_k/(2+2\lip(\pi))} \) and \( \cal{L}_{C_k,U_k,\nu_k/(2+2\lip(\pi))} \).
			
			Since \( U_k \) is disjoint from \( \cN(\supp(\mu),\nu_k/4) \) and \( \mu_{X_k} \) converges weakly to \( \mu \), there exists \( N(k)<\i \) so that if \( j\geq N(k) \) then 
			\begin{equation}
				\label{thm:haircut:eqa}
				\H^m(X_j\cap U_k)<\l^{-1}\mu_{X_j}(U_k)<\inf \left\{\frac{\nu_k}{\k \lip(\pi)^m \cal{L}_{C_k,U_k,\nu_k/(2+2\lip(\pi))}},\,\,\frac{\nu_k}{\k},\,\,\cal{K}_{C_k,U_k,\nu_k/(2+2\lip(\pi))}\right\},
			\end{equation}
			and 
			\begin{equation}
				\label{thm:haircut:eqb}
				X\subset \cN(X_j, \nu_k/4).
			\end{equation}
			Let \( Y_k \) be the set produced from \( X_{N(k)+k} \) by Theorem \ref{thm:poof}, which we may apply by \eqref{thm:haircut:eqa}. In particular,
			\begin{equation}
				\label{thm:haircut:eqc}
				Y_k\cap B(X,\nu_k/4)=X_{N(k)+k}\cap B(X,\nu_k/4),
			\end{equation}
			\begin{equation}
				\label{thm:haircut:eq0}
				Y_k\subset \cN(X_{N(k)+k},\nu_k/(2+2\lip(\pi)))\cap C_k^c\subset \cN(X,\nu_k/2)\subset U,
			\end{equation}
			and
			\begin{equation}
				\label{thm:haircut:eq1}
				\H^m(Y_k\cap C^c) \leq \H^m(Y_k\cap U_k) \leq \cal{L}_{C_k,U_k,\nu_k/2} \H^m(X_{N(k)+k}\cap U_k)<\frac{\nu_k}{\k \lip(\pi)^m}.
			\end{equation}
			It follows from Lemma \ref{lem:6A} and Corollary \ref{cor:corespans} that \( \hat{Y}_k\equiv (\pi(Y_k)\setminus A)^*\cup A \) is an element of \( \Span(A,C,L,m) \) and from \eqref{thm:haircut:aq0} and \eqref{thm:haircut:aq1} that \( \hat{Y}_k=\supp(\mu_{\hat{Y}_k})\cup A \). Moreover, by \eqref{thm:haircut:eq0} the points of \( Y_k\cap C^c \) lie no further than \( \nu_k/(2+2\lip(\pi)) \) from \( C \), so again by \eqref{thm:haircut:eq0},
			\begin{equation}
				\label{thm:haircut:eqz}
				\hat{Y}_k\subset \cN(Y_k,\nu_k/2)\subset \cN(X,\nu_k).
			\end{equation}
			
			Since \( Y_k=(Y_k\setminus A)^*\cup A \), it follows from \eqref{thm:haircut:eqc} that \( X_{N(k)+k}\cap B(X,\nu_k/4)\subset \hat{Y}_k\cap B(X,\nu_k/4) \). It follows from \eqref{thm:haircut:eqb} that \( X\subset \cN(\hat{Y}_k, \nu_k/4). \) Therefore by \eqref{thm:haircut:eqz}, the sequence \( \{\hat{Y}_k\}_k \) converges to \( X \) in the Hausdorff metric. 
			
			Observe that
			\begin{equation}
				\label{thm:haircut:eqoo}
				(\pi(Y_k\cap C^c))^c \cap (\hat{Y}_k \setminus A)\subset Y_k \cap C
			\end{equation}
			and that by the definition of \( \hat{Y}_k \) and the Vitali covering theorem,
			\begin{equation}
				\label{thm:haircut:eqo}
				\H^m(Y_k \cap C \cap (\pi(Y_k\cap C^c))^c \setminus \hat{Y}_k )=0.
			\end{equation}
			
			Suppose \( \cal{B}\subset \R^n \) is a Borel set. It follows from \eqref{thm:haircut:eq1} and \eqref{thm:haircut:eqoo} that
			\begin{equation}
				\label{thm:haircut:ezy}
				\mu_{\hat{Y}_k}(\cal{B})\leq \mu_{\hat{Y}_k}(\cal{B}\cap (\pi(Y_k\cap C^c))^c\cap C) + \k \lip(\pi)^m \H^m(Y_k\cap C^c) <\mu_{\hat{Y}_k}(\cal{B}\cap (\pi(Y_k\cap C^c))^c\cap C) + \nu_k.
			\end{equation}
			By \eqref{thm:haircut:eqc}, \eqref{thm:haircut:eq1}, and \eqref{thm:haircut:ezy},
			\begin{align}
				\mu_{\hat{Y}_k}(\cal{B}) &< \mu_{\hat{Y}_k}(\cal{B}\cap (\pi(Y_k\cap C^c))^c\cap C) + \nu_k \notag\\
				&=\mu_{\hat{Y}_k}(\cal{B}\cap (\pi(Y_k\cap C^c))^c\cap C \cap U_k) + \mu_{X_{N(k)+k}}(\cal{B} \cap U_k^c) + \nu_k \notag\\
				&\leq \k\H^m(Y_k\cap U_k) + \mu_{X_{N(k)+k}}(\cal{B}) + \nu_k \notag\\
				&\leq 2\nu_k + \mu_{X_{N(k)+k}}(\cal{B}).\label{thm:haircut:eqt}
			\end{align}
			In particular, this proves that \( \sup_k \{ \mu_{\hat{Y}_k}(\R^n)\}<\i \), and so there exists a subsequence \( k_i\to \i \) so that \( \mu_{\hat{Y}_{k_i}} \) converges weakly to a finite Borel measure \( \eta \). It remains to show that \( \eta=\mu \).

			Let \( \cal{Q} \) be a dyadic subdivision of \( \R^n \) such that for each cube \( Q\in\cal{Q} \), \( \eta(\fr\, Q)=\mu(\fr\, Q)=0 \). Let \( Q\in\cal{Q} \). By the Portmanteau theorem and \eqref{thm:haircut:eqt}, \[ \eta(Q)=\lim_{i\to\infty} \mu_{\hat{Y}_{k_i}}(Q) \leq \lim_{i\to\infty} 2\nu_{k_i}+\mu_{X_{N(k_i)+k_i}}(Q)=\mu(Q). \] Likewise, by \eqref{thm:haircut:eqa}, \eqref{thm:haircut:eqc}, \eqref{thm:haircut:eq1} and \eqref{thm:haircut:eqo},
			\begin{align*}
				\mu(Q) &= \lim_{i\to\infty} \mu_{X_{N(k_i)+k_i}}(Q)\\
				&= \lim_{i\to\infty}\mu_{X_{N(k_i)+k_i}}(Q\cap B(X,\nu_{k_i}/4)) + \mu_{X_{N(k_i)+k_i}}(Q\cap U_{k_i})\\
				&\leq \liminf_{i\to\infty}\mu_{\hat{Y}_{k_i}}(Q\cap B(X,\nu_{k_i}/4) \cap (\pi(Y_{k_i}\cap C^c))^c\cap C) + \k(1+\lip(\pi)^m)\H^m(Y_{k_i}\cap C^c) + \nu_{k_i}\\
				&\leq \lim_{i\to\infty}\mu_{\hat{Y}_{k_i}}(Q)+3\nu_{k_i}\\
				&= \eta(Q).
			\end{align*}
			If \( W\subset \R^n \) is open, by taking a Whitney decomposition of \( W \) using cubes from \( \cal{Q} \) we conclude that \( \mu(W)=\eta(W) \). Since both measures are Radon, outer regularity proves they are equal, and so the sequence \( X_i'\equiv \hat{Y}_{k_i} \) has the desired properties.
		\end{proof}
		
		Lemma \ref{lem:21C} implies
		\begin{cor}
			\label{cor:closedinweaktop}
			Under the assumptions of Theorem \ref{thm:haircut}, the set \( X \) is a surface with coboundary \( \supset L \).
		\end{cor}
		
\section{Minimizing sequences}
	\label{sec:minimizing_sequences}
	\setcounter{thm}{0}
	\subsection{H\"older densities}
		\label{sec:elliptic_integrands}

		For the rest of this paper, assume that \( (C,A) \) is compact, \( m\geq 2 \), and that \( L\neq \emptyset \). Assume also that \( C \) is a Lipschitz neighborhood retract\footnote{Indeed, ultimately we will assume that \( C \) is convex, but the results in this section hold in greater generality.} (letting \( \pi:U\to C \) denote the retraction) and \( f:C\to [a,b] \) is \( \a \)-H\"older, with \( 0<a\leq b<\i \). Finally, assume that \( \Span(A,C,L,m) \) is non-empty (e.g. if \( \H^{m-1}(A)<\i \) and \( p\in \R^n \), then by Lemmas \ref{lem:2A} and \ref{lem:6A} the image under the map \( \phi \) of the cone \( C_p A \) is a surface with cobounadry \( \supset L \).)  

		For the sake of notation, extend \( f \) to all of \( \R^n \) by setting \( f\equiv 0 \) on \( C^c \) and again let \( \F^m \) denote the measure \( f\sp^m \) on \( \R^n \).
		
		\begin{defn}
			\label{def:minimizingseq}
			Let \[ \frak{m} \equiv \inf \{ \F^m(X\setminus A) : X\subset C \textrm{ is a compact surface with coboundary }\supset L \}. \] By Corollary \ref{cor:bddbelow}, \( \frak{m}>0 \). If \( \{X_k\}_{k\in\N}\subset \Span(A,C,L,m) \) such that 
			\begin{enumerate}
				\item \( X_k=\supp(\F^m\lfloor{X_k\setminus A})\cup A \) for all \( k \),
				\item \( \F^m(X_k\setminus A)\to \frak{m} \),
				\item \( \{\F^m\lfloor{X_k\setminus A}\} \) converges weakly to a finite Borel measure \( \mu_0 \), and
				\item \( \{X_k\} \) converges to \( X_0\equiv \supp(\mu_0)\cup A \) in the Hausdorff metric,
			\end{enumerate}
			then we say \( \{X_k\} \) is a \emph{\textbf{convergent minimizing sequence}}. 
		\end{defn}
	
		By Lemma \ref{lem:21C}, if \( \{X_k\} \) is a convergent minimizing sequence, then \( X_0 \) is a compact surface with coboundary \( \supset L \), contained in \( C \), and \( \mu_0(\R^n)=\frak{m} \). It follows from the Riesz representation theorem, the Banach-Alaoglu theorem, and Theorem \ref{thm:haircut} that there exists a convergent minimizing sequence.
	
		We will show, in the following order, that if \( \{X_k\}_{k\in\N}\subset \Span(A,C,L,m) \) is a convergent minimizing sequence, then:
		\begin{enumerate}
			\item There exists a subsequence of \( \{X_k\} \) which is Reifenberg regular (see Definition \ref{def:rregular} below,)
			\item The lower density of \( \mu_0 \) is bounded away from zero, uniformly across all points in \( X_0\setminus A \),
			\item \( \H^m(X_0\setminus A)<\i \), hence \( X_0 \in \Span(A,C,L,m) \),
			\item The density of \( \mu_0 \) exists, is non-zero and is finite for every \( p\in X_0\setminus A \), hence
			\item \( X_0\setminus A \) is \( m \)-rectifiable. Finally, we prove
			\item \( \F^m(X_0\setminus A)=\frak{m} \).
		\end{enumerate}

	\subsection{Reifenberg regular sequences}
		\label{sub:lower_bounds_on_density_ratios}
		For every set \( X\subset \R^n \), and \( W\subset \R^n \) let \( \G(X,W,R) \) denote the set of closed balls whose centers lie in \( X \), that are contained in \( W \), and whose radii are bounded above by \( R \).

		\begin{defn}
			\label{def:rregular}
			Let \( 0<c<\i \) and \( 0<R\leq \i \). A sequence \( \{Y_k\} \) of subsets of \( \R^n \) is \emph{\textbf{Reifenberg (\( (c,R) \)-)regular in \( W \) (in dimension \( m \))}} if \[ \H^m(Y_k(p,r)) \ge c r^m \] for all \( k \ge 1 \), \( r > 2^{-k} \) and \( B(p,r) \in \G(Y_k,W,R) \). If \( \{Y_k\} \) is Reifenberg \( (c,R) \)-regular in \( W \) for some \( 0<c<\i \) and \( 0<R\leq \i \), we say that \( \{Y_k\} \) is simply \emph{\textbf{Reifenberg regular}} in \( W \).
		\end{defn}
		
		Note that the property of being Reifenberg \( (c,R) \)-regular is stable under taking subsequences. We show in the next two lemmas that there exists a constant \( 0<\mathbf{c}<\i \) depending on \( n, a, b \) and \( C \), and a constant \( \mathbf{R} \) depending on \( C \) such that every convergent minimizing sequence \( \{X_k\}_{k\in\N}\subset \Span(A,C,L,m) \) has a subsequence which is Reifenberg \( (\mathbf{c},\mathbf{R}) \)-regular in \( A^c \).
		
		Let \( \mathbf{R}>0 \) be small enough so that for any \( x\in C \), the ball of radius \( R \) about \( x \) is contained in \( U \). Let \( 0<M<\i \) be a constant large enough so that
		\begin{equation}
			\label{lem:epsilonbound:eq:a}
			2^{M+m-2}\geq \frac{b}{a}\lip(\pi)^m.
		\end{equation}
						
		\begin{lem}
			\label{lem:epsilonbound}
			If \(\{X_k\}_{k\in\N}\subset \Span(A,C,L,m) \) is a convergent minimizing sequence, then there exists a subsequence \( k_i\to\i \) so that
			\begin{equation}
				\label{lem:epsilonbound:2}
				\sp^m(X_{k_i}(p,r)) > \frac{2^{-(i+1)m}}{2\left(2^M m\right)^m\left(2^{n+M-1} K_m^n\right)^{m-1}} \,\,\, \mbox{ for all }\,\,\, B(p,r) \in \G(X_{k_i},A^c,\mathbf{R}) \mbox{ and } r>2^{-(i+1)}.
			\end{equation}
		\end{lem}
			
		\begin{proof}
			If not, there exist \( N_1<\i \) and \( N_2 >0 \) such that for all \( k\geq N_1\), there exists \( B(p_k,r_k)\in \G(X_k,A^c,\mathbf{R}) \) with 
			\begin{equation}
				\label{eq:rkawayfromA}
				r_k>2^{-(N_2+1)}
			\end{equation}
			and
			\begin{equation}
				\label{eq:neededforhaircut}
				\sp^m(X_k(p_k,r_k)) \leq \frac{2^{-(N_2+1)m}}{2\left(2^M m\right)^m\left(2^{n+M-1} K_m^n\right)^{m-1}}.
			\end{equation}
			Since \( X_k \) converges to \( X_0 \) in the Hausdorff metric, there exists a point \( p\in X_0 \) and a subsequence \( p_{k_j}\to p \). By \eqref{eq:rkawayfromA}, the distance from \( p_k \) to \( A \) is bounded below, hence \( p\notin A \). 
						
			By Lemma \ref{lem:haircut}, there exist \( 2^{-N_2-1-M} < r_{k_j}' < r_{k_j} \) and a compact surface \( Y_j \) with coboundary \( \supset L \) satisfying 
			\begin{equation}
				\label{eq:epbound0}
				Y_j \cap (\cN(p_{k_j},r_{k_j}')^c) = X_{k_j} \cap (\cN(p_{k_j},r_{k_j}')^c)
			\end{equation}
			and
			\begin{equation}
				\label{eq:epbound17}
				\sp^m(Y_j(p_{k_j},r_{k_j}')) \leq \frac{1}{2^{M+m-1}} \sp^m(X_{k_j}(p_{k_j},r_{k_j}')).
			\end{equation}
			
			Let \( Z_j=\pi(Y_j(p_{k_j},r_{k_j}')) \). Then by \eqref{lem:epsilonbound:eq:a} and \eqref{eq:epbound17},
			\begin{equation}
				\label{eq:epbound1}
				\F^m(Z_j) \leq \frac{1}{2} \F^m(X_{k_j}(p_{k_j},r_{k_j}')).
			\end{equation}
			Since \( \mu_0(\R^n)<\i \), we may fix \( 0<r<2^{-N_2-1-M} \) so that \( \mu_0(\fr\, B(p,r))=0 \). For \( j \) sufficiently large, \( B(p,r)\subset B(p_{k_j},r_{k_j}') \). The Portmanteau theorem gives \[ \lim_{j\to\i} \F^m(X_{k_j}(p,r))=\mu_0(B(p,r)). \] Since \( p\in \supp(\mu_0) \),
			\begin{equation}
				\label{eq:epbound2}
				0 < \frac{1}{2} \mu_0(B(p,r)) < \F^m(X_{k_j}(p,r)) \leq \F^m(X_{k_j}(p_{k_j}, r_{k_j}'))
			\end{equation}
			for \( j \) sufficiently large. Using \eqref{eq:epbound0}, \eqref{eq:epbound1} and \eqref{eq:epbound2} we deduce
			\begin{align*}
				\F^m(\pi(Y_j)\setminus A)&\leq \F^m(Z_j)+\F^m(X_{k_j}\cap (B(p_{k_j},r_{k_j}')^c)\cap (A^c))\\
				&\le \frac{1}{2} \F^m(X_{k_j}(p_{k_j},r_{k_j}')) + \F^m(X_{k_j}\cap B(p_{k_j},r_{k_j}')^c \cap A^c)\\
				&=\F^m(X_{k_j}\setminus A)-\frac{1}{2} \F^m(X_{k_j}(p_{k_j},r_{k_j}'))\\
				&<\F^m(X_{k_j}\setminus A)-\frac{1}{4} \mu_0(B(p,r)).
			\end{align*}
			Since \( \F^m(X_{k_j}\setminus A)\to \frak{m} \), we have \( \F^m(\pi(Y_j)\setminus A)<\frak{m} \) for \( j \) sufficiently large, a contradiction since \( \pi(Y_j)\in\Span(A,C,L,m) \).
		\end{proof}
	
		\begin{lem}
			\label{lem:prelrmin}
			There exists \( 0<\mathbf{c}<\i \) depending only on \( n, a, b \) and \( C \) such that if \( \{X_k\}_{k\in\N}\subset \Span(A,C,L,m) \) is a convergent minimizing sequence, then there exists a subsequence \( k_i \to \i \) such that \( \{X_{k_i}\} \) is Reifenberg \( (\mathbf{c},\mathbf{R}) \)-regular in \( A^c \).
		\end{lem}
		
		\begin{proof}
			Since \( \F^m(X_k\setminus A)\to \frak{m} \), let us assume without loss of generality that
			\begin{equation}
				\label{eq:prel1}
				\F^m(X_k\setminus A)\leq \frak{m} + a\frac{2^{-(k+1)m-1}}{2\left(2^M m\right)^m\left(2^{n+M-1} K_m^n\right)^{m-1}}.
			\end{equation}
			Let \( \{X_{k_i}\} \) be the subsequence determined by Lemma \ref{lem:epsilonbound}. We will show that this subsequence is Reifenberg regular. By Lemma \ref{lem:4} it suffices to find constants \( \mathbf{c} > 0 \) and \( 0< \mathbf{R}\leq \i \) such that \[ \int_0^r \sp^{m-1} (x_{k_i}(p,t)) \,dt \geq \mathbf{c} r^m \] for all \( i\geq 1 \), \( r > 2^{-i} \) and \( B(p,r) \in \G(X_{k_i},A^c,\mathbf{R}) \). 
				
			Fix \( i\geq 1 \), \( r>2^{-i} \) and \( B(p,r)\in \G(X_{k_i},A^c,\mathbf{R}) \). For every \( s\in(2^{-(i+1)},r) \), we have by Lemma \ref{lem:epsilonbound},
			\begin{equation}
				\label{eq:prel2}
				\F^m(X_{k_i}(p,s)) > a\frac{2^{-(i+1)m}}{2\left(2^M m\right)^m\left(2^{n+M-1} K_m^n\right)^{m-1}}.
			\end{equation} 

			 Let \( Z_{k_i} \) denote the image under the retraction \( \pi:U\to C \) of the set obtained from Lemma \ref{lem:8} applied to \( x_{k_i}(p,s) \). In particular,
			\begin{equation}
				\label{prop:prel5}
				\F^m(Z_{k_i}) \leq b\lip(\pi)^m K_m^n \,(\sp^{m-1}(x_{k_i}(p,s)))^{m/(m-1)}.
			\end{equation}
 			Lemmas \ref{lem:13} and \ref{lem:6A} imply \( \F^m(Z_{k_i}\cup (X_{k_i}\setminus B(p,s)))\geq \frak{m} \), and since \( i\leq k_i \), it follows from \eqref{eq:prel1} that
			\begin{equation}
				\label{eq:prel}
				\F^m(X_{k_i}(p,s)) - a\frac{2^{-(i+1)m-1}}{2\left(2^M m\right)^m\left(2^{n+M-1} K_m^n\right)^{m-1}} \leq \F^m(Z_{k_i}). 
			\end{equation}
				
			Thus, by Lemma \ref{lem:4}, \eqref{eq:prel2}, \eqref{eq:prel}, and \eqref{prop:prel5},
			\begin{align*}
				a\int_0^s \sp^{m-1} (x_{k_i}(p,t)) \,dt &\leq \F^m((X_{k_i}(p,s)) \\
				&< 2\F^m(X_{k_i}(p,s)) - a\frac{2^{-(i+1)m}}{\left(2^M m\right)^m\left(2^{n+M-1} K_m^n\right)^{m-1}} \\
				&\leq  2\F^m(Z_{k_i}) \\
				&\leq 2b\lip(\pi)^m K_m^n \,(\sp^{m-1}(x_{k_i}(p,s)))^{m/(m-1)}.
			\end{align*}
			In other words, by the Lebesgue Differentiation Theorem, for almost every \( s\in(2^{-(i+1)},r), \) we have
			\[ \frac{\frac{d}{ds}\int_0^s \sp^{m-1} (x_{k_i}(p,t)) \,dt}{\left(\int_0^s \sp^{m-1} (x_{k_i}(p,t)) \,dt\right)^{(m-1)/m}}\geq \left(2\lip(\pi)^m\frac{b}{a}K_m^n\right)^{-(m-1)/m}. \] 
 
			Integrating, this implies \[ \left(\int_0^r\sp^{m-1} (x_{k_i}(p,t))\,dt\right)^{1/m}-\left(\int_0^{2^{-(i+1)}}\sp^{m-1} (x_{k_i}(p,t))\,dt\right)^{1/m} \ge \frac{r-2^{-(i+1)}}{m}\left(2\lip(\pi)^m\frac{b}{a}K_m^n\right)^{-(m-1)/m}, \] and since \( r\geq 2^{-i} \),
				
			\[ \int_0^r \sp^{m-1} (x_{k_i}(p,t))	\,dt \geq \frac{r^m}{(2m)^m\left(2\lip(\pi)^m\frac{b}{a}K_m^n\right)^{m-1}}. \]
			
			Therefore, the constant \( \mathbf{c}=(2m)^{-m}\left(2\lip(\pi)^m\frac{b}{a}K_m^n\right)^{-(m-1)} \) satisfies the desired conditions.
		\end{proof}

		\begin{remark}
			Lemmas \ref{lem:epsilonbound} and \ref{lem:prelrmin} hold for sequences \( \{X_k\} \) which minimize measures more general than \( \F^m \), in particular those of the form described in Theorem \ref{thm:haircut}.
		\end{remark}
		
	 \subsection{General properties of Reifenberg regular sequences}
		\begin{defn}
			\label{def:beta}
			For \( \{Y_k\}_{k \in \N} \) a sequence of subsets of \( \R^n \), subsets \( Y_0, W \) of \( \R^n \), and \( 0<R\leq \i \), let \[ \b(\{Y_k\},Y_0,W,R,m) \equiv \inf\left\{\liminf_{k \to \i} \frac{\H^m(Y_k(p,r))}{\a_m r^m} : B(p,r) \in \G(Y_0,W,R) \right\}. \] If \( \G(Y_0,W,R) \) is empty, let \( \b(\{Y_k\},Y_0,W,R,m)=0 \).
		\end{defn}  

		\begin{prop}
			\label{prop:lowerbound}
			Let \( Y_0 \) and \( W \) be subsets of \( \R^n \). If \( \{Y_k\}_{k\in \N} \) is a sequence of subsets of \( \R^n \) which is Reifenberg \( (c,R) \)-regular in \( W \), and every \( p\in Y_0\cap W \) is the limit of a sequence \( p_k\subset Y_k \), then \[ \b(\{Y_k\}, Y_0, W,R,m) \geq c/\a_m > 0. \]
		\end{prop}
		
		\begin{proof}
			Suppose \( B(p,r) \in \G(Y_0,W,R) \) and \( 0<\d<1 \). For sufficiently large \( k \), there exists \( B(p_k,r_k)\in \G(Y_k,W,R) \) such that \[ 2^{-k}<\d \cdot r < r_k < r \] and \( B(p_k,r_k)\subset B(p,r) \). By Definition \ref{def:rregular}, \[ \H^m(Y_k(p,r))\geq  \H^m(Y_k(p_k,r_k))\geq c\, r_k^m >c\, \d^m\, r^m, \] and thus \( \b(\{Y_k\}, Y_0, W,R,m)\geq c \,\d^m/\a_m \). Now let \( \d \to 1 \).
		\end{proof}

		If \( \mu \) is a Borel measure, denote its upper and lower \( m \)-dimensional densities at \( p \) by \[ \Theta^*_m(\mu,p) \equiv \limsup_{r \to 0} \frac{\mu(B(p,r))}{\a_m r^m}, \] and \[ {\Theta_*}_m(\mu,p) \equiv \liminf_{r \to 0} \frac{\mu(B(p,r))}{\a_m r^m}, \] respectively. If both quantities are equal, the \( m \)-dimensional density of \( \mu \) at \( p \) is denoted by \[ \Theta_m(\mu,p) \equiv \lim_{r \to 0} \frac{\mu(B(p,r))}{\a_m r^m}. \]
	
		\begin{prop}
			\label{prop:lowerbound2}
			Suppose \( \{\nu_k\}_{k\in \N} \) is a sequence of Radon measures on \( \R^n \) such that \( \nu_k \) converges weakly to some Radon measure \( \nu_0 \). Suppose also that \( \{Y_k\}_{k\in \N} \) is a sequence of subsets of \( \R^n \), \( W\subset \R^n \) is open, \( \cal{Q}\geq 0 \), and \( \nu_k\lfloor_W\geq \cal{Q} \H^m\lfloor_{W\cap Y_k} \) for all \( k\geq 1 \). If \( Y_0\subset \R^n \), \( p\in Y_0\cap W \) and \( r>0 \) is small enough so that \( B(p,r)\in \G(Y_0,W,R) \), then \( \nu_0(B(p,r))/(\a_m r^m) \ge \cal{Q} \b(\{Y_k\},Y_0,W,R,m) \). In particular, \( {\Theta_*}_m(\nu_0,p) \ge \cal{Q} \b(\{Y_k\},Y_0,W,R,m) \).
		\end{prop}

		\begin{proof}
			By the Portmanteau theorem,
			\begin{align*}
				\frac{\nu_0(B(p,r))}{\a_m r^m} &\ge \limsup_{k \to \i} \frac{\nu_k(B(p,r))}{\a_m r^m}\\
				&\ge \cal{Q} \liminf_{k \to \i} \frac{\H^m(Y_k(p,r))}{\a_m r^m}\\
				&\ge \cal{Q} \b(\{Y_k\},Y_0,W,R,m).
			\end{align*}
		\end{proof}
		
		Combining Propositions \ref{prop:lowerbound} and \ref{prop:lowerbound2},
		
		\begin{cor}
			\label{cor:abstractlowerbound}
			Suppose \( \{\nu_k\}_{k\in \N} \) is a sequence of Radon measures on \( \R^n \) such that \( \nu_k \) converges weakly to some Radon measure \( \nu_0 \). Let \( Y_k=\supp(\nu_k) \) for all \( k\geq 0 \). Suppose \( W\subset \R^n \) is open, \( \cal{Q}>0 \) and \( \nu_k\lfloor_W\geq \cal{Q} \H^m\lfloor_{W\cap Y_k} \) for all \( k\geq 1 \). If \( \{Y_k\}_{k\in \N} \) is Reifenberg \( (c,R) \)-regular in \( W \), then \( \nu_0(B(p,r))/(\a_m r^m) \ge \cal{Q} \b(\{Y_k\},Y_0,W,R,m)\geq \cal{Q}c/\a_m>0 \) for all \( p \in Y_0\cap W \) and \( r>0 \) small enough so that \( B(p,r)\in \G(Y_0,W,R) \).
		\end{cor}
		
		In particular, 
		\begin{cor}
			\label{cor:lowerbound}
			If \( \{X_k\}_{k\in \N}\subset \Span(A,C,L,m) \) is a convergent minimizing sequence which is Reifenberg \( (\mathbf{c},\mathbf{R}) \)-regular in \( A^c \), then \( \mu_0(B(p,r))/(\a_m r^m) \ge a \mathbf{c}/\a_m>0 \) for all \( p\in X_0\setminus A \) and \( 0<r<R \) small enough so that \( B(p,r)\subset A^c \). In particular, \( {\Theta_*}_m(\mu_0,p) \ge a \mathbf{c}/\a_m>0. \)
		\end{cor}
		
		\begin{thm}
			\label{thm:generalbddhsdorffmeasure}
			Let \( W\subset \R^n \) be open, and let \( Y_0\subset \R^n \). If \( \{Y_k\}_{k\in \N} \) is a sequence of subsets of \( \R^n \) which is Reifenberg \( (c,R) \)-regular in \( W \), and every \( p\in Y_0\cap W \) is the limit of a sequence \( p_k\in Y_k \), and if \( K=\liminf_{k\to \i}\H^m (Y_k\cap W)<\i \), then \( \H^m(Y_0\cap W)\leq 5^m \frac{K}{\b(\{Y_k\},Y_0,W,R,m)}\leq 5^m K \a_m/c <\i \).
		\end{thm}

		\begin{proof}
			Suppose \( \{B(p_i, r_i)\}_{i\in I} \subset \G(Y_0,W,R) \) is a collection of disjoint balls. If \( J\subset I \) is finite, then by the definition of \( \b(\{Y_k\},Y_0,W,R,m) \),
			\begin{align*}
				\b(\{Y_k\},Y_0,W,m)\sum_{j\in J} \a_m r_j^m &\leq \sum_{j\in J} \liminf_{k \to \i} \H^m(Y_k(p_j,r_j))\\
				&\leq \liminf_{k \to \i} \sum_{j\in J}  \H^m(Y_k(p_j,r_j))\\
				&\leq \liminf_{k \to \i} \H^m (Y_k\cap W)\\
				&=K.
			\end{align*}
			Since \( I \) is necessarily countable, and since by Proposition \ref{prop:lowerbound} we may divide by \( \b(\{Y_k\},Y_0,W,R,m) \),
			\begin{equation}
				\label{44}
				\a_m \sum_{i\in I}r_i^m \leq \frac{K}{\b(\{Y_k\},Y_0,W,R,m)}.
			\end{equation}
			Now fix \( \d \) and \( \d' \) such that \( \d' > \d > 0 \). The subcollection of \( \G(Y_0,W,R) \) consisting of balls of radius \( r<\d \) covers \( Y_0 \cap \cN(W^c,\d')^c \). So, by the Vitali Covering Lemma and \eqref{44}, it follows that \[ \H_{10\d}^m(Y_0 \cap \cN(W^c,\d')^c) \leq 5^m \frac{K}{\b(\{Y_k\},Y_0,W,R,m)}. \] Letting \( \d \to 0 \) and then \( \d' \to 0 \), we deduce from Proposition \ref{prop:lowerbound} that \[  \H^m(Y_0\cap W) \leq 5^m \frac{K}{\b(\{Y_k\},Y_0,W,R,m)}\leq 5^m K \a_m/c <\i. \]
		\end{proof}
		
		By Proposition \ref{prop:lowerbound},
		\begin{cor}
			\label{cor:finite}
			Suppose \( \{X_k\} \) is a convergent minimizing sequence which is Reifenberg \( (\mathbf{c},\mathbf{R}) \)-regular in \( A^c \). Then \[ \H^m(X_0\setminus A) \leq 5^m \frac{\frak{m}}{\b(\{X_k\},X_0, A^c,R,m) a}\leq 5^m \frak{m} \a_m/\mathbf{c} < \i. \]
		\end{cor}
		
		\begin{defn}
			If \( p\in \R^n \), \( E\in \Gr(m,n) \) and \( 0<\e<1 \), let \( \cal{C}(p,E,\e) = \{q \in \R^n: d(q-p, E) < \e d(p,q) \} \). Let \( T_{p,r}(x)=(x-p)/r \). We say that a Radon measure \( \mu \) has an \emph{\textbf{approximate tangent space \( E \) with multiplicity \( \theta\in \R \) at \( p \)}} if \( r^{-m}{T_{p,r}}_* \mu \) converges weakly to \( \theta\H^m\lfloor_E \) as \( r\to 0 \).
		\end{defn}

		\begin{prop}
			\label{lem:recmeasure}
			Suppose \( \{\nu_k\}_{k\in \N} \) is a sequence of Radon measures on \( \R^n \) such that \( \nu_k \) converges weakly to some Radon measure \( \nu_0 \). Suppose also that \( \{Y_k\}_{k\in \N} \) is a sequence of subsets of \( \R^n \), \( W\subset \R^n \) is open, \( \cal{Q}>0 \), and \( \nu_k\lfloor_W\geq \cal{Q} \H^m\lfloor_{W\cap Y_k} \) for all \( k\geq 1 \). Suppose \( Y_0\subset \R^n \) and \( \b(\{Y_k\},Y_0,W,R,m)>0 \) for some \( 0<R\leq \i \). If \( p\in W \) and \( \nu_0 \) has an approximate tangent space \( E\in \Gr(m,n) \) with multiplicity \( \theta>0 \) at \( p \), then for every \( 0<\e<1 \) there exists \( r>0 \) such that \( Y_0(p,r)\setminus \overline{\cal{C}(p,E,\e)}=\emptyset \).
		\end{prop}

		\begin{proof}
			If the result is false, there exists a sequence \( p_i\to p \) with \( p_i\in Y_0\setminus \overline{\cal{C}(p,E,\e)} \). Let \( r_i=2d(p_i,p) \). We have \( B(p_i, \e r_i/4)\subset B(p,r_i)\setminus \overline{\cal{C}(p,E,\e/2)} \). For \( i \) large enough, \( B(p_i, \e r_i/4)\in \G(Y_0,W,R) \), so
			\begin{align*}
				(\e/4)^m \a_m r_i^m \b(\{Y_k\},Y_0,W,R,m) &\leq \frac{1}{\cal{Q}}\liminf_{k\to \i} \nu_k(B(p_i,\e r_i/4))\\
				&\leq \frac{1}{\cal{Q}} \nu_0(B(p_i,\e r_i/4))\\
				&\leq \frac{1}{\cal{Q}} \nu_0(B(p,r_i)\setminus \overline{\cal{C}(p,E,\e/2))}.
			\end{align*}
		
			The Portmanteau theorem gives
			\begin{align*}
				\H^m\lfloor_E(B(0,1)\setminus \cal{C}(0,E,\e/2)) &\geq \theta^{-1}\limsup_{i\to \i} r_i^{-m}{T_{p,r_i}}_*\nu_0(B(0,1)\setminus \cal{C}(0,E,\e/2))\\
				&\geq \theta^{-1}\cal{Q} (\e/4)^m \a_m \b(\{Y_k\},Y_0,W,R,m)\\
				>0,
			\end{align*}
			a contradiction since \( m>0 \).
		\end{proof}
		
		Combining this with Proposition \ref{prop:lowerbound}, we deduce
		\begin{cor}
			\label{cor:recmesasure}
			Suppose \( \{\nu_k\}_{k\in \N} \) is a sequence of Radon measures on \( \R^n \) such that \( \nu_k \) converges weakly to some Radon measure \( \nu_0 \). Let \( Y_k=\supp(\nu_k) \) for all \( k\geq 0 \). Suppose \( W\subset \R^n \) is open, \( \cal{Q}>0 \) and \( \nu_k\lfloor_W\geq \cal{Q} \H^m\lfloor_{W\cap Y_k} \) for all \( k\geq 1 \). If \( \{Y_k\}_{k\in \N} \) is Reifenberg regular in \( W \), \( p\in W \) and \( \nu_0 \) has an approximate tangent space \( E\in \Gr(m,n) \) with multiplicity \( \theta>0 \) at \( p \), then for every \( 0<\e<1 \) there exists \( r>0 \) such that \( Y_0(p,r)\setminus \overline{\cal{C}(p,E,\e)}=\emptyset \).
		\end{cor}
		
		In particular,
		\begin{cor}
			\label{cor:recmeasures2}
			If \( \{X_k\}_{k\in \N}\subset \Span(A,C,L,m) \) is a convergent minimizing sequence which is Reifenberg regular in \( A^c \), \( p\in A^c \) and \( \mu_0 \)  has an approximate tangent space \( E\in \Gr(m,n) \) with multiplicity \( \theta>0 \) at \( p \), then for every \( 0<\e<1 \) there exists \( r>0 \) such that \( X_0(p,r)\setminus \overline{\cal{C}(p,E,\e)}=\emptyset \).
		\end{cor}
 		
\section{Regularity}
	\label{sec:monotonicity}
	\subsection{Rectifiability}
	\label{sub:rectifiability}
	In this section we will use Preiss' density theorem (\cite{preiss}, see also \cite{mattila} Theorem 17.8) to prove that \( X_0\setminus A \) is \( m \)-rectifiable. For the rest of the paper let us assume \( C \) is convex. Hence, we may take \( \mathbf{R}=\i \).

	Let \( \{X_k\}_{k\in \N}\subset \Span(A,C,L,m) \) be a convergent minimizing sequence which is Reifenberg \( (\mathbf{c},\i) \)-regular in \( A^c \). Fix \( p \in X_0\setminus A \) and let \( d_p \) denote the distance between \( p \) and \( A \). Let \( g_0(r) = \mu_0(B(p,r)) \), and for \( k\geq 1 \), let \( g_k(r) = g_k(p,r) =\F^m(X_k(p,r)) \).

	\begin{defn}
		\label{def:dppp}
		Let \( D_p \) be the subset of \( (0,d_p) \) consisting of numbers \( r \) such that the following conditions hold for all \( s = r/2^N \) with \( N \ge 0 \):
		\begin{enumerate}
			\item\label{def:dppp:item:1} \( \mu_0(\fr B(p,s)) = 0 \),
			\item\label{def:dppp:item:2} \( \sp^{m-1}(x_k(p,s))<\i \) for all \( k\geq 1 \),
			\item\label{def:dppp:item:3} \( g_k \) is differentiable at \( s \), for all \( k\geq 0 \),
			\item\label{def:dppp:item:4} \( \lim_{h\to 0} \int_s^{s+h}\left(\int_{x_k(p,t)} f(q)d\sp^{m-1}(q)\right) dt/h = \int_{x_k(p,s)} f(q)d\sp^{m-1}(q)\ \) for all \( k\geq 1 \).
	 
		\end{enumerate}
	\end{defn}

	\begin{lem}
		\label{lem:prime}
		\( D_p \) is a full Lebesgue measure subset of \( (0,d_p) \).
	\end{lem}  

	\begin{proof}
		Part \ref{def:dppp:item:1} determines a full Lebesgue measure set since \( \mu_0 \) is finite. Part \ref{def:dppp:item:2} follows from Lemma \ref{lem:4}. Part \ref{def:dppp:item:3} follows since \( g_k \) is monotone non-decreasing. Part \ref{def:dppp:item:4} follows from the Lebesgue Differentiation Theorem.
	\end{proof}

	Since \( \F^m(X_k) \to \frak{m} \), there exists a decreasing sequence \( \e_k \to 0 \) such that
	\begin{align}
		\label{eqek}
		\F^m(X_k) \le \frak{m} + \e_k.
	\end{align}
	\begin{lem}
		\label{ldt}
		If \( r \in D_p \), then \( \F^{m-1}(x_k(p,r)) \le g_k'(r) \) for all \( k\geq 1 \).
	\end{lem}

	\begin{proof}
		By Definition \ref{def:dppp} and Lemma \ref{lem:4},
		\begin{align*}
			g_k'(r) &= \lim_{h\to 0} \frac{g_k(r+h)- g_k(r)}{h}\\
					&= \lim_{h\to 0} \frac{\int_{X_k(p,r+h)\setminus X_k(p,r)}f(q) d\sp^m(q)}{h}\\
					&\ge \lim_{h\to 0} \int_r^{r+h}\left(\int_{x_k(p,t)} f(q)d\sp^{m-1}(q)\right)\,dt/h\\
					&= \int_{x_k(p,r)} f(q)d\sp^{m-1}(q)\\
					&= \F^{m-1}(x_k(p,r)).
		\end{align*}
	\end{proof}

	\begin{lem}
		\label{lem:4F}\mbox{}
		If \( r \in D_p \), then
		\begin{equation}
			\label{eq:lem4f}
			g_k(r) \le \frac{r}{m} \left(1 + 2|f|_{C^{0,\a}} r^\a/a\right) g_k'(r) + \e_k
		\end{equation}
		for all \( k\geq 1 \).
	\end{lem}

	\begin{proof}
		The cohomological version\footnote{For the proof, replace Lemmas 11A and 15A of \cite{reifenberg} with Lemmas \ref{lem:11A} and \ref{lem:15A}.} of Lemma 7 of \cite{reifenberg} and \eqref{eqek} imply
		\begin{align*}
			\F^m(X_k(p,r)) &\leq (f(p)+|f|_{C^{0,\a}} r^\a)\left( \frac{r}{m} \sp^{m-1}(x_k(p,r))\right) +\e_k\\
			&\leq \frac{r}{m} \F^{m-1}(x_k(p,r)) + 2|f|_{C^{0,\a}} \frac{r^{1+\a}}{m} \sp^{m-1}(x_k(p,r)) + \e_k\\
			&\leq \frac{r}{m} \F^{m-1}(x_k(p,r)) + 2|f|_{C^{0,\a}} \frac{r^{1+\a}}{am} \F^{m-1}(x_k(p,r)) + \e_k.
		\end{align*}
		Now apply Lemma \ref{ldt}.
	\end{proof}
 
	\begin{lem}
		\label{lem:ratios}
		There exist constants \( 0 < k_p \le 1 \) and \( \d_p \in (0,d_p) \) such that if \( r \in (0,\d_p) \cap D_p \), then\footnote{In the case that \( f \) is \( \a \)-H\"{o}lder continuous for \( 0<\a<1 \), the right hand side of \eqref{eq:lemratios} is replaced with \( k_p^{r^\a} \).} 
		\begin{equation}
			\label{eq:lemratios}
			\frac{g_0(r)/r^m}{g_0(s)/s^m} \ge k_p^r
		\end{equation}
		for all \( s \in (0,r]\cap D_p \).
	\end{lem}
	
	\begin{proof}
		By Corollary \ref{cor:lowerbound} there exists \( \d_p > 0 \) such that if \(0 < r < \d_p \), then
		\begin{align}
			\label{lem:3*}
			g_k(r) \ge (a\mathbf{c}/2)r^m
		\end{align}
		for sufficiently large \( k \) depending on \( r \).

		Suppose \( r/2 \le t \le r \). By \eqref{lem:3*}, since \( g_k \) is monotone nondecreasing,
		\begin{align}
			\label{gkr}
			g_k(t) \ge g_k(r/2) \ge  (a\mathbf{c}/2)(r/2)^m
		\end{align}
		for sufficiently large \( k \) depending on \( r \).

		It follows from Lemmas \ref{lem:4F} and \ref{ldt} that if \( t \in D_p \), then
		\begin{align*}
			1 \le \frac{t}{m} (1 + 2|f|_{C^{0,\a}} r^\a/a) ) g'_k(t))/g_k(t)) + \e_k/g_k(r/2)).
		\end{align*}
		By \eqref{gkr},
		\begin{align}
			\label{eq:fracg}
			\frac{g_k'(t)}{g_k(t)} &\ge \frac{m}{t}\left(1- \frac{\e_k}{(a\mathbf{c}/2)(r/2)^m}\right)\frac{1}{1+2|f|_{C^{0,\a}} r^\a/a}\\
			&\ge \frac{m}{t}\left(1- \frac{\e_k}{(a\mathbf{c}/2)(r/2)^m}\right)(1-2|f|_{C^{0,\a}} r^\a/a)\notag\\
			&= \frac{m}{t}((1-2|f|_{C^{0,\a}} r^\a/a) - \e_k \psi(r)),\notag
		\end{align}
		where \( \psi(r) = (1-2|f|_{C^{0,\a}} r^\a/a)/((a\mathbf{c}/2)(r/2)^m) \). 
		
		Let \( r/2 \le s^* \le r \). By \eqref{eq:fracg},
		\begin{align*}
			\log(g_k(r)/g_k(s^*)) &\ge \int_{s^*}^r \log(g_k)'(t) dt\\
			&= \int_{s^*}^r g_k'(t)/g_k(t)dt\\
			&\ge \int_{s^*}^r \frac{m}{t}((1-2|f|_{C^{0,\a}} r^\a/a) - \e_k \psi(r)) dt\\
			&\ge \log\left(\frac{r^m}{{s^*}^m}\right) \left((1-2|f|_{C^{0,\a}} r^\a/a) - \e_k \psi(r)\right)\\
			&\ge \left[\log\left(\frac{r^m}{{s^*}^m}\right)\right]-\left[m\log(4)|f|_{C^{0,\a}} r^\a/a)\right] - \left[m\log(2)\e_k \psi(r) \right].
		\end{align*} 
		In other words, \[ \frac{g_k(r)/r^m}{g_k({s^*})/{s^*}^m} \ge \exp(-m\log(4)|f|_{C^{0,\a}} r^\a/a) \exp(-m\log(2)\e_k \psi(r)) \] for sufficiently large \( k \) depending on \( r \).
		
		If \( s^* \in D_p \), letting \( k \to \i \), we get
		\begin{align*}
			\frac{ g_0(r)/r^m}{g_0(s^*)/{s^*}^m} &\ge \exp(-m\log(4)|f|_{C^{0,\a}} r^\a/a)  = k_p^{r/2}
		\end{align*}
		where \( k_p = \exp(-m\log(8)|f|_{C^{0,\a}} /a)\). 

		Let \( N\ge 0 \) satisfy \( r/2^{N+1} < s \le r/2^N \). Then
		\begin{align*}
			\frac{g_0(r)/r^m}{g_0(s)/s^m} &= \left(\frac{g_0(r)/r^m}{g_0(r/2)/(r/2)^m}\right) \left(\frac{g_0(r/2)/(r/2)^m}{g_0(r/4)/(r/4)^m}\right) \cdots \left(\frac{g_0(r/2^N)/(r/2^N)^m}{g_0(s)/s^m}\right) \\
			&\ge k_p^{(r/2)(1 + 1/2 + 1/4 + \cdots +1/2^N)} \\
			& = k_p^{(r/2)(2 - 1/2^N)} \\
			&\ge k_p^{r}
		\end{align*}
		since \( k_p \le 1 \).
	\end{proof}

	\begin{thm}[Density]
		\label{thm:density}
		The density \( \Theta_m(\mu_0,p) \) exists, is finite and uniformly bounded above zero.
	\end{thm}

	\begin{proof}
		Let \( h_0(r) = \frac{g_0(r)}{\a_mr^m} \). By Corollary \ref{cor:lowerbound} and Lemmas \ref{lem:ratios} and \ref{lem:prime}, we know that \[ 0<a\mathbf{c}/\a_m<\liminf_{r\to 0} h_0(r)\leq \limsup_{r\to 0} h_0(r)<\i. \] Suppose \( \liminf_{r \to 0} h_0(r) < \limsup_{r \to 0} h_0(r) \). Then there exist sequences \( s_i < r_i \),  \( r_i \to 0 \), \( s_i \to 0 \), \( h_0(r_i) \to \liminf_{r \to 0} h_0(r) \) and \( h_0(s_i) \to \limsup_{r \to 0} h_0(r) \). There exist \( r_i', s_i' \in D_p \) with \( s_i \le s_i' \le r_i' \le r_i \), \( s_i'/s_i \to 1 \), and \( r_i'/r_i \to 1 \). By Lemma \ref{lem:ratios},
		\begin{align*}
			1 &\le \liminf_{i \to \i}\frac{g_0(r_i')/(\a_m(r_i')^m)}{g_0(s_i')/(\a_m(s_i')^m)} \\
			&\le \liminf_{i \to \i}\frac{g_0(r_i)/(\a_mr_i^m)}{g_0(s_i)/(\a_m s_i^m)}\\
			&= \liminf_{i \to \i} \frac{h_0(r_i)}{h_0(s_i)}\\
			&= \frac{\liminf_{r \to 0} h_0(r)} {\limsup_{r \to 0} h_0(r)} < 1,
		\end{align*}
		a contradiction.
	\end{proof}
 
	\begin{cor}
		\label{cor:rectifiableX}
		\( X_0 \setminus A \) is \( m \)-rectifiable.
	\end{cor}

	\begin{proof}
		By Preiss' theorem, the measure \( \mu_0\lfloor_{A^c} \) is \( m \)-rectifiable, and so there exists an \( m \)-rectifiable Borel set \( E \) such that \( \mu_0\lfloor_{A^c}(E^c)=0 \). In particular, by Corollary \ref{cor:lowerbound}, \[ 0=\mu_0\lfloor_{A^c}(X_0\setminus (A\cup E))=\mu_0((X_0\setminus A)\setminus E)\geq a \mathbf{c}/\a_m \H^m((X_0\setminus A)\setminus E), \] showing \( X_0\setminus A \) is \( m \)-rectifiable.
	\end{proof}

\subsection{Lower semicontinuity} 
	\label{sub:lower_semicontinuity}
	Let \( C_m = 2m(2\a_m (b/a))^{1/m} (2b/(a\mathbf{c}))^{1/m} \).
	
	\begin{lem}
		\label{lem:halfradius} 
		Suppose \( \Theta_m(\mu_0,p) < 2b \). There exists \( 0 < \d_p' \) such that for each \( r \in (0,\d_p')\cap D_p \) there exists \( N_{p,r} \) with the property that each \( J \subset (r/2,r) \) with full Lebesgue measure contains some \( r' \) satisfying \[ r' \sp^{m-1}(x_k(p,r')) \le C_m \sp^m(X_k(p,r')) \] for all \( k \ge N_{p,r} \).
	\end{lem}

	\begin{proof}
		By the Portmanteau theorem and Corollary \ref{cor:lowerbound}, there exists \( 0 < \d_p' \) such that for each \( r\in (0,\d_p')\cap D_p \) there exists \( N_{p,r} \) such that
		\begin{equation}
			\label{gkrbound}
			a\mathbf{c} r^m/2 < g_k(r) < 2\a_m b r^m
		\end{equation}
		for each \( k \ge N_{p,r} \).
		
		We show there exists \( r'\in J \) such that
		\begin{equation}
			\label{mminusone}
			\sp^{m-1}(x_k(p,r')) \le 2m(2\a_m(b/a))^{1/m}\left(\int_0^{r'}\sp^{m-1}(x_k(p,t))dt\right)^{(m-1)/m}
		\end{equation}
		for all \( k \ge N_{p,r} \).
		
		Let \( I_k(p,s):=\int_0^s\sp^{m-1}(x_k(p,t))dt \). By Lemma \ref{lem:4} and \eqref{gkrbound},
		\begin{align}
			\label{jkmx}
			I_k(p,r) < 2\a_m(b/a) r^m
		\end{align}
		for all \( r\in (0,\d_p')\cap D_p \) and all \( k\ge N_{p,r} \).

		Suppose \eqref{mminusone} fails. By the Lebesgue Differentiation Theorem, \[ \frac{\frac{d}{ds}I_k(p,s)}{I_k(p,s)^{(m-1)/m}} > 2m (2\a_m(b/a))^{1/m} \] for almost every \( r/2 < s < r \). Integrating, this implies
		\begin{align}
			\label{lem:halfradius:eq:1}
			\int_{r/2}^r \frac{d}{ds}\left(I_k(p,s)^{1/m}\right) ds > r (2\a_m(b/a))^{1/m}.
		\end{align}
		Since \( s\mapsto I_k(p,s) \) is increasing and absolutely continuous, the function \( s\mapsto I_k(p,s)^{1/m} \) is also absolutely continuous, and so \( I_k(p,r)^{1/m} - I_k(p,r/2)^{1/m} > r(2\a_m(b/a))^{1/m} \), contradicting \eqref{jkmx}. This establishes \eqref{mminusone}.

		By \eqref{mminusone} and Lemma \ref{lem:4},
		\begin{align}
			\label{eq:spm-1}
			r'\sp^{m-1}(x_k(p,r'))\le (2\a_m(b/a)(2m)^m)^{1/m}\,r'\sp^m(X_k(p,r'))^{(m-1)/m}.
		\end{align}
		for all \( k \ge N_{p,r} \).

		We have \( (r')^m \sp^m(X_k(p,r'))^{m-1} < 2b/(a\mathbf{c}) \sp^m(X_k(p,r'))^m \) by \eqref{gkrbound}. Thus, \eqref{eq:spm-1} gives
		\begin{align*}
			r' \sp^{m-1}(x_k(p,r')) \le 2m(2\a_m (b/a))^{1/m} (2b/(a\mathbf{c}))^{1/m} \sp^m(X_k(p,r'))=C_m \sp^m(X_k(p,r'))
		\end{align*}
		for all \( k \ge N_{p,r} \).
	\end{proof}

	Recall the constants \( \d_p \) from Lemma \ref{lem:ratios} and \( \d_p' \) from Lemma \ref{lem:halfradius}.
	
	\begin{thm}
		\label{thm:lsc}
		For each open set \( V\subset A^c \), \[ \F^m(X_0 \cap V) \le \liminf_{k\to\i} \F^m(X_k \cap V). \]
	\end{thm}
	
	\begin{proof}
		\quad\\ \medskip 
		  \textbf{Initial setup:}
		 \quad\\ \medskip
				For \( k\geq 0 \), let \( X_k'=X_k\cap V \). Let \( X_0^\second \) be the subset of \( X_0' \) consisting of those points \( q \) such that \( \Theta_m(\H^m\lfloor_{X_0'}, q)=1 \) and for which \( \mu_0 \) has an approximate tangent space \( E_q \) with multiplicity \( \theta_q>0 \) at \( q \). The set \( X_0' \) is \( m \)-rectifiable by Corollary \ref{cor:rectifiableX}, and \( \H^m(X_0') < \i \) by Corollary \ref{cor:finite}. Furthermore, the measure \( \mu_0\lfloor_V \) is \( m \)-rectifiable by Theorem \ref{thm:density}, \cite{ambrosio} so\footnote{See \cite{ambrosio} Theorem 2.83(i).} \( \H^m(X_0'\setminus X_0^\second)=0 \).
				
				Let \( \e>0 \) and let
				\begin{equation}
					\label{eq:choiceofepsilon}
					\zeta = \min\left\{1/2, \frac{a\a_{m-1}}{b2^{2m+1}\a_m C_m}, \e \right\}.
				\end{equation} 
				
				Let \( p\in X_0^\second \). Since \( \Theta_m(\H^m\lfloor_{X_0'}, p)=1 \) and since \( X_0' \) is \( m \)-rectifiable\footnote{Hence \( \H^m \) and \( \sp^m \) are equal, see \cite{federer}.}, there exists \( 0<\k_p<\i \) such that
				\begin{equation}
					\label{eq:lastonesurely}
					\sp^m(X_0'(p,r))\leq (1+\zeta)\a_mr^m
				\end{equation}
				for all \( 0<r<\k_p \).
				
				Fix \( p\in X_0^\second \) and suppose \( \Theta_m(\mu_0,p)<2b \). The constant \( \d_p' \) from Lemma \ref{lem:halfradius} is defined, so by Lemma \ref{lem:prime} there exist, for every \( r\in (0,\d_p')\cap D_p \), a constant \( N_{p,r} \) and some \( s_{p,r}\in (r/2,r)\cap D_p \) such that 
				\begin{equation}
					\label{eq:lolzz}
					s_{p,r} \sp^{m-1}(x_k(p,s_{p,r})) \leq C_m \sp^m(X_k(p,s_{p,r}))
				\end{equation}
				for all \( k \ge N_{p,r} \).
		
				By Corollary \ref{cor:recmeasures2}, there exists \[ 0<\d_p^\second \le \min\{ \d_p', \mathrm{dist}(p,V^c), \zeta, \k_p \} \] such that if \( 0<r<\d_p^\second \) then 
				\begin{equation}
					\label{x0prime}
					X_0'(p,r)\subset \cal{N}(p+E_p, r \zeta).
				\end{equation}
		
				By Corollary \ref{cor:lowerbound}, by decreasing \( \d_p^\second \) if necessary, we may also require that
				\begin{align}
					\label{eq:c0r}
					a\mathbf{c}/2 < \mu_0(B(p,r))/r^m
				\end{align}
				for all \( 0<r < \d_p^\second \).
		
				Since \( X_k\to X_0 \) in the Hausdorff metric, for every \( 0<r<\d_p^\second \) we may choose the constant \( N_{p,r} \) large enough so that
				\begin{equation}
					\label{xkinnbhd}
					X_k'(p,s_{p,r}) \subset \cal{N}(p+E_p,\zeta s_{p,r})
				\end{equation}
				for all \( k>N_{p,r} \).

				Thus if \( 0<r<\d_p^\second \) and \( k>N_{p,r} \), Lemma \ref{lem:12} implies
				\begin{equation}
					\label{eq:12Alsc}	
					\sp^m(X_k'(p,s_{p,r})) + \zeta s_{p,r}\,\frac{2^{2m}\a_m}{\a_{m-1}}\,\sp^{m-1}(x_k'(p,s_{p,r})) \geq \a_m s_{p,r}^m,
				\end{equation}
				or there exists a surface \( Q_{k,p,r} \subset \fr\, B(p,s_{p,r}) \cap \cal{N}(p+E_p,\zeta s_{p,r}) \)\\ with coboundary \( \supset K^*(X_k'(p,s_{p,r}),x_k'(p,s_{p,r})) \) such that
				\begin{equation}
					\label{eq:12Blsc}
					\sp^m(Q_{k,p,r}) \leq \zeta s_{p,r}\,\frac{2^{2m}\a_m}{\a_{m-1}}\,\sp^{m-1}(x_k'(p,s_{p,r})).
				\end{equation}
				We show that \eqref{eq:12Alsc} holds for sufficiently large \( k \).   

				Suppose there exists a sequence \( k_i \to \i \) and surfaces \( Q_{k_i,p,r} \) satisfying \eqref{eq:12Blsc}. Let \( \hat{X}_{k_i} = (X_{k_i}\cap B(p,s_{p,r})^c) \cup Q_{{k_i},p,s_{p,r}} \). By \eqref{eq:12Blsc}, \eqref{eq:lolzz}, and \eqref{eq:choiceofepsilon},
				\begin{align*}
					\F^m(\hat{X}_{k_i}(p,s_{p,r})) &\le b \sp^m(\hat{X}_{k_i}(p,s_{p,r}))\\
					&= b \sp^m(Q_{k_i,p,r})\\
					&\leq b\zeta s_{p,r}\, \frac{2^{2m}\a_m}{\a_{m-1}}\, \sp^{m-1}(x_{k_i}'(p,s_{p,r}))\\
					&\leq b\zeta \frac{2^{2m}\a_m}{\a_{m-1}}  C_m \sp^m(X_{k_i}'(p,s_{p,r}))\\
					&\leq \frac{1}{2}\F^m(X_{k_i}'(p,s_{p,r})),
				\end{align*}
				the right hand side of which is bounded below for large enough \( i \) by \( a\mathbf{c} s_{p,r}^m/4 \) by the Portmanteau theorem and \eqref{eq:c0r}.
		
				On the other hand, \( \hat{X}_{k_i} \) is a surface with coboundary \( \supset L \) by Lemma \ref{lem:13}, so in particular, \( \F^m(\hat{X}_{k_i})\geq \frak{m} \). This gives a contradiction with \eqref{eqek} when \( i \) is large enough so that \( \e_{k_i}< a\mathbf{c} s_{p,r}^m/4 \).

				Thus, if \( 0<r<\d_p^\second \), we may choose the constant \( N_{p,r} \) large enough so that \eqref{eq:12Alsc} holds for all \( k>N_{p,r} \) as long as \( \Theta_m(\mu_0,p) < 2b \).
				\medskip
				\quad\\ \medskip 
			 \textbf{A lower bound on \( \F^m(X_k'(p,s_{p,r})) \):} 
			 \quad\\ \medskip 
								For \( p\in X_0^\second \) with \( \Theta_m(\mu_0,p)<2b \), \( r \in (0,\d_p^\second)\cap D_p \) and \( k>N_{p,r} \), we have by \eqref{eq:12Alsc},   \eqref{eq:lolzz} and since \( s_{p,r}\leq \zeta  \), 
				  \begin{align}\label{ffm} 
					  \begin{split}
	 				   \F^m(X_k'(p,s_{p,r})) &= \int_{X_k'(p,s_{p,r})} f(q) d\H^m(q) \\
	 				   						 &\ge (f(p) -|f|_{C^{0,\a}} s_{p,r}^\a )  \H^m(X_k'(p,s_{p,r}))\\
	 										 &\ge (f(p) -|f|_{C^{0,\a}} \zeta^\a ) \a_m s_{p,r}^m C_m'(\zeta) 
					  \end{split}
				  \end{align}
	where \( C_m'(\zeta) = 1/ \left(1 + \zeta C_m \,\frac{2^{2m}\a_m}{\a_{m-1}}\right) \).
	Observe that \( C_m'(\zeta) \uparrow 1 \) as \( \zeta \to 0 \).  
 				Now consider  \( p\in X_0^\second \) with \( \Theta_m(\mu_0,p)\ge 2b \) and let \( 0<\d_p^\second\leq \min \{\mathrm{dist}(p,V^c),\k_p,\zeta\} \) be small enough so that \( \mu_0(B(p,r))\ge b\a_m r^m \) for all \( r\in (0,\d_p^\second)\cap D_p \).  By the Portmanteau theorem, for each \( r\in (0,\d_p^\second)\cap D_p \) there exists \( N_{p,r}<\i \) such that
				 \begin{align*}
				  \F^m(X_k'(p,r))\ge b\a_m r^m
				 \end{align*} 
				 for all \( k > N_{p,r}. \)
				 For each \( r\in (0,\d_p^\second) \) let \( s_{p,r}=r \). We obtain
 				\begin{equation}
 					\label{eq:2bornot2b}
 					\F^m(X_k'(p,s_{p,r}))\ge f(p) \a_m s_{p,r}^m
 				\end{equation}
				 for all \( k > N_{p,r}. \)

			We may assume \( \zeta \) is so small that \( |f|_{C^{0,\a}} \zeta^\a/a < 1 \).   Putting   \eqref{ffm}  and \eqref{eq:2bornot2b} together we have
			\begin{align} \label{fffm2bornot2b}
				\begin{split}
				\F^m(X_k'(p,s_{p,r})) &\ge  \min\{f(p) \a_m s_{p,r}^m,\,  C_m'(\zeta) f(p)\a_m s_{p,r}^m  -  C_m'(\zeta)|f|_{C^{0,\a}} \zeta^\a \a_m s_{p,r}^m\}\\ 
				&= f(p) \a_m s_{p,r}^m \min\{1,  C_m'(\zeta)(1 - |f|_{C^{0,\a}} \zeta^\a/f(p))\}\\
				& \ge f(p) \a_m s_{p,r}^m \min\{1, C_m'(\zeta) \}\\
				& \ge f(p) \a_m s_{p,r}^m C_m'(\zeta) 				
				\end{split}
			\end{align}
			for all \( p\in X_0^\second \), \( r\in (0,\d_p^\second)\cap D_p \) and \( k>N_{p,r} \).
			
				\quad\\ \medskip  
			 \textbf{An upper bound on \( \F^m(X_0'(p,r)) \):} \\
			 	\quad\\ \medskip 
				We next compare \( \F^m(X_0'(p,r)) \) with \( f(p)\a_mr^m \). For \( p\in X_0^\second \) and \( 0<r<\d_p^\second \), it follows from \eqref{eq:lastonesurely} that
				\begin{align}
					\label{x0est}
					\F^m(X_0'(p,r)) &= \int_{X_0'(p,r)} f(q) d \sp^m(q) \\
  					&\le (f(p) + |f|_{C^{0,\a}} r^\a) \sp^m(X_0'(p,r))\notag\\
					&\le (f(p) + |f|_{C^{0,\a}} \zeta^\a) (1+\zeta)\a_mr^m.\notag
				\end{align}
  			  \quad\\ \medskip 
			 \textbf{A Vitali-type covering:}
			 \quad\\ \medskip 
				 By Corollary \ref{cor:finite}, there exists a covering \( \{B(p_i,s_{p_i,r_i})\}_{i\in I} \) of \( \sp^m \) almost all \( X_0' \) by disjoint balls \( B(p_i,s_{p_i,r_i}) \) with \( p_i \in X_0^\second \) and \( r_i \in (0,\d_{p_i}^\second)\cap D_{p_i} \).  The conditions needed for Theorem 2.8 of \cite{mattila} thus hold and we obtain a Vitali-type covering for \( \F^m \). In particular, there exists a finite covering \( \{B(p_i,s_{p_i,r_i})\}_{i=1}^N \) such that
				\begin{align}
					\label{coveringbyballs}
					\F^m(X_0'\setminus \cup_{i=1}^N X_0'(p_i,s_{p_i,r_i})) < \e.
				\end{align}
				Let \( N = \max_i \{N_{p_i,r_i}\} \). By \eqref{coveringbyballs}, \eqref{x0est}, \eqref{fffm2bornot2b},  and since \( \zeta\leq \e \),
				\begin{align*}
					\F^m(X_0') &\le \sum_i \F^m(X_0'(p_i,s_{p_i,r_i})) + \e\\
					&\le \sum_i (f(p_i) + |f|_{C^{0,\a}} \zeta^\a) (1+\zeta)\a_ms_{p_i,r_i}^m + \e\\
					&\le  (1 + |f|_{C^{0,\a}} \zeta^\a/a) (1+\zeta)\sum_if(p_i)\a_ms_{p_i,r_i}^m + \e\\
					&\le  (1 + |f|_{C^{0,\a}} \zeta^\a/a) (1+\zeta)\sum_i\F^m(X_k'(p_i,s_{p_i,r_i}))/C_m'(\zeta)  + \e\\
					&\le  (1 + |f|_{C^{0,\a}} \e^\a/a) (1+\e) \F^m(X_k')/C_m'(\e)  + \e
				\end{align*}
for all \( k \ge N \). The result follows by letting \( \e \to 0 \).
 	\end{proof}

	In particular, \( \F^m(X_0\setminus A) = \frak{m} \). 
	
	The regularity statement of our main theorem follows from the main result in \cite{almgren}: Let \( 1 \le \g < \i \) and \( 0 < \d < \i \). A locally compact surface \( S \) with finite \( \H^m \)-measure is \( (\g, \d) \)-\emph{restricted} with respect to a closed set \( A \) if \( S \subset \R^n \setminus A \) and \( \H^m(S \cap W) \le \g \H^m(\phi(S \cap W)) \) whenever \( \phi:\R^n \to \R^n \) is Lipschitz and \( W \equiv \{x \in \R^n: \phi(x) \ne x\} \) satisfies \( W \cap A = \phi(W)\cap A=\emptyset \) and \( \diam(W \cup \phi(W )) < \d \). 
	
	Now suppose \( \o:\R^+ \to \R^+ \) is nondecreasing with \( \lim_{r \downarrow 0} \o(r) = 0 \), \( 0<\d <\i \) and \( A \subseteq \R^n \) is closed. A set \( S \) is \( (M, \o, \d) \) \emph{minimal} with respect to \( A \) if \( S \) is \( (\g,\d) \)-restricted with respect to \( A \) for some \( \g > 0 \) and the following condition is satisfied:
	If \( \phi:\R^n \to \R^n \) is Lipschitz such that \( r = \diam(W \cup \phi(W)) \le \d \), then \[ \H^m((S \cap W)) \le (1+ \o(r))\H^m(\phi(S \cap W)). \]
	 
	 We will show that \( X_0\setminus A \) is \( (M,\o,\d) \) minimal with respect to \( A \), for some \( \o \) and \( \d \). Note that \( X_0\setminus A \) is \( (b/a, \d) \)-restricted with respect to \( A \) for all \( \d>0 \).
	 
	 Fix \( \d< (a/2|f|_{C^{0,\a}})^{1/\a} \) let \( \o(r)=4|f|_{C^{0,\a}} r^\a/a \). Suppose \( \phi:\R^n \to \R^n \) is Lipschitz such that \( r = \diam(W \cup \phi(W)) \le \d \).

	 Choose a point \( p \in (X_0\setminus A) \cap W \). Since \( \diam(W)\le r \), and since \( X_0 \) is \( \F^m \)-minimizing, and \( \phi \) is the identity outside of \( W \), we have 
	 \begin{align*}
		 \H^m((X_0\setminus A) \cap W) &\le \F^m((X_0\setminus A) \cap W)/(f(p) - |f|_{C^{0,\a}}r^\a)\\
		 &\le \F^m(\phi((X_0\setminus A) \cap W))/(f(p) - |f|_{C^{0,\a}}r^\a)\\
		 &= \int_{\phi((X_0\setminus A) \cap W)} f(q)d\H^m(q)/(f(p) - |f|_{C^{0,\a}}r^\a)\\
		 &\le \int_{\phi((X_0\setminus A) \cap W)} (f(p) + |f|_{C^{0,\a}} r^\a) d\H^m(q)/(f(p) - |f|_{C^{0,\a}}r^\a)\\
		 &\le (1+ 4|f|_{C^{0,\a}} r^\a/a) \H^m(\phi((X_0\setminus A) \cap W)).
	 \end{align*}
	 We conclude that \( X_0\setminus A \) satisfies the conditions for H\"older regularity in \cite{almgren} IV.(13)(6). In particular, there exists an open set \( U \subset \R^n \) such that \( \H^m((X_0 \setminus A) \setminus U) = 0 \) and \( (X_0 \setminus A) \cap U \) is a locally H\"older continuously differentiable submanifold of \( \R^n \) with exponent \( \a \).

	This completes the proof of our main theorem.  
 	
	\addcontentsline{toc}{section}{References} 
	\bibliography{bibliography.bib}{}
	\bibliographystyle{amsalpha}
 
\end{document}